\documentclass[11pt,reqno,a4paper]{amsart}

\usepackage{fullpage}
\usepackage{setspace}
\usepackage{datetime}
\usepackage{verbatim}
\usepackage{enumerate}
\usepackage{color}
\usepackage{url}
\usepackage{dsfont}
\usepackage[dvipsnames]{xcolor}

\usepackage{amsmath}
\usepackage{cite}
\usepackage{euflag}

\usepackage{graphicx,caption,epstopdf}

\usepackage[shortlabels]{enumitem} 

\allowdisplaybreaks

\usepackage{etoolbox} 

\usepackage[mathlines]{lineno}
\usepackage{etoolbox} 

\newcommand*\linenomathpatch[1]{%
   \expandafter\pretocmd\csname #1\endcsname {\linenomath}{}{}%
   \expandafter\pretocmd\csname #1*\endcsname{\linenomath}{}{}%
   \expandafter\apptocmd\csname end#1\endcsname {\endlinenomath}{}{}%
   \expandafter\apptocmd\csname end#1*\endcsname{\endlinenomath}{}{}%
 }
\newcommand*\linenomathpatchAMS[1]{%
    \expandafter\pretocmd\csname #1\endcsname {\linenomathAMS}{}{}%
    \expandafter\pretocmd\csname #1*\endcsname{\linenomathAMS}{}{}%
    \expandafter\apptocmd\csname end#1\endcsname {\endlinenomath}{}{}%
    \expandafter\apptocmd\csname end#1*\endcsname{\endlinenomath}{}{}%
}

\expandafter\ifx\linenomath\linenomathWithnumbers
\let\linenomathAMS\linenomathWithnumbers
\patchcmd\linenomathAMS{\advance\postdisplaypenalty\linenopenalty}{}{}{}
\else
\let\linenomathAMS\linenomathNonumbers
\fi

\linenomathpatchAMS{gather}
\linenomathpatchAMS{multline}
\linenomathpatchAMS{align}
\linenomathpatchAMS{alignat}
\linenomathpatchAMS{flalign}
\linenomathpatch{equation}

\newtheorem{theorem}             {Theorem}[section]
\newtheorem{lemma}     [theorem] {Lemma}

\newtheorem{definition}[theorem] {Definition}   
\newtheorem{proposition}[theorem] {Proposition}   

\newtheorem{claim}[theorem] {Claim}
\newtheorem{remark}[theorem] {Remark}

\newtheorem{proto-theo}[theorem] {Proto-Theorem}   

\def\canarrow{\xrightarrow[{\raisebox{.5mm}[1mm][0mm]{$\scriptstyle \rm
			$}}]{\raisebox{0.0mm}[0mm]{$\scriptstyle \rm can$}}}

\def\wcanarrow{\xrightarrow[{\raisebox{.5mm}[1mm][0mm]{$\scriptstyle \rm
			$}}]{\raisebox{0.0mm}[0mm]{$\scriptstyle \rm wcan$}}}

\let\:=\colon

\def\bbn{{\mathbb N}}
\def\NN{{\mathbb N}}

\def\bfg{{\mathbf G}}

\def\call{{\mathcal L}}

\let\cH\calh
\let\cP\calp
\let\cG\calg

\let\cC\calc
\let\cL\call

\let\epsilon\varepsilon
\let\eps\epsilon

\def\Pr{\mathop{\mathbb{P}}\nolimits}
\def\Ex{\mathop{\mathbb{E}}\nolimits}
\def\Var{\mathop{\text{\rm Var}}\nolimits}

\def\ex{\mathop{\text{\rm ex}}\nolimits}

\long\def\OLD#1{}

\let\subset\subseteq

\usepackage[colorlinks=false]{hyperref}

\begin{document}
	\pagestyle{plain}
	\thispagestyle{empty}
	\footskip=30pt
	\shortdate
	\settimeformat{ampmtime}
	\onehalfspacing

	\title[A canonical Ramsey theorem for even cycles in random graphs]%
	{A canonical Ramsey theorem for even cycles in random graphs}
	
	\author[José~D.~Alvarado]{José~D.~Alvarado}
        \address[J. D. Alvarado]{Faculty of Mathematics and Physics, University of Ljubljana, Jadranska 19, Ljubljana, Slovenia}
        \email{jose.alvarado@fmf.uni-lj.si}
	
	\author[Yoshiharu Kohayakawa]{Yoshiharu Kohayakawa}
	\address[Y. Kohayakawa]{Instituto de Matem\'atica e Estat\'{\i}stica \\Universidade de 
		S\~ao Paulo, Rua do Mat\~ao 1010\\05508--090~S\~ao Paulo, Brazil}
	\email{yoshi@ime.usp.br}
	
	\author[Patrick Morris]{Patrick Morris}
	\address[P. Morris]{Departament de Matem\`atiques, Universitat
          Polit\`ecnica de Catalunya (UPC), Carrer de Pau Gargallo 14, 08028  Barcelona, Spain}
	\email{pmorrismaths@gmail.com}

	\author[Guilherme~O.~Mota]{Guilherme~O.~Mota}
	\address[G. O. Mota]{Instituto de Matem\'atica e Estat\'{\i}stica \\Universidade de 
		S\~ao Paulo, Rua do Mat\~ao 1010\\05508--090~S\~ao Paulo, Brazil}
	\email{mota@ime.usp.br}
		
	\thanks{J. D. Alvarado was partially supported by FAPESP
		(2020/10796-0). Y. Kohayakawa was partially supported by CNPq
		(407970/2023-1, 420838/2025-2, 315258/2023-3) and FAPESP
		(2023/03167-5). P. Morris was supported by the Deutsche Forschungsgemeinschaft (DFG, German
Research Foundation) Walter Benjamin programme - project number 504502205 and by the European Union's Horizon Europe   Marie Sk{\l}odowska-Curie grant RAND-COMB-DESIGN - project number
101106032 {\euflag}. G. O. Mota was partially supported by CNPq
		(315916/2023-0, 420838/2025-2 ) and FAPESP (2024/13859-4, 2023/03167-5). This study was financed in part by CAPES, Coordenação
		de Aperfeiçoamento de Pessoal de Nível Superior, Brazil, Finance
		Code~001.  FAPESP is the S\~ao Paulo Research Foundation.  CNPq is
		the National Council for Scientific and Technological Development of
		Brazil.}

	\date{\today, \currenttime}

	\begin{abstract}
		The celebrated canonical Ramsey theorem of Erd\H{o}s and Rado implies that for $2\leq k\in \NN$, any colouring of the edges of $K_n$ with $n$ sufficiently large gives a copy of $C_{2k}$ which has one of three canonical colour patterns: monochromatic, rainbow or lexicographic. In this paper we show that if $p=\omega(n^{-1+1/(2k-1)}\log n)$, then $\bfg(n,p)$ will asymptotically almost surely also have the property that any colouring of its edges induces canonical copies of $C_{2k}$. This determines the threshold for the canonical Ramsey property with respect to even cycles, up to a $\log$ factor. 
	\end{abstract}
	
	\maketitle
	\onehalfspacing

	\section{Introduction}
	\label{sec:Introduction}
	
	Ramsey theory, in broad terms, is the mathematical study of
        finding inevitable order in a large chaotic universe. In this
        paper, our setting will be to consider all edge colourings of
        some large host graph $G$ and the inevitable order will be
        copies of some fixed graph $H$ in $G$ that exhibit certain
        \emph{canonical} colour patterns.  
	The colour patterns of interest are the following.
	
	\begin{definition}[Canonical colour pattern] \label{def:canonical copy}
		Let $H$ be a graph, $\sigma$ an ordering of $V(H)$ and  $\chi:E(H)\rightarrow \bbn$ an edge colouring of $H$. We define  $H$ coloured by $\chi$ to be \emph{canonical} (with respect to $\sigma$) if  one of the following holds
		\begin{enumerate}[{\rm(i)}]
			\item \label{item:mono} $H$ is \emph{monochromatic}: $\chi(e)=\chi(f)$ for all $e,f\in E(H)$;
			\item \label{item:rbow} $H$ is \emph{rainbow}: $\chi(e)\neq \chi(f)$ for all $e\neq f\in E(H)$;
			\item \label{item:lex} $H$ is \emph{lexicographic} (with respect to $\sigma$): there exists an injective assignment of colours $\phi:V(H)\rightarrow \bbn$ such that for every edge $e=uv\in E(H)$ with $u<_\sigma v$, we have that $\chi(e)=\phi(u)$. 
		\end{enumerate}
	\end{definition}
	
	Examples of all the canonical patterns of $C_6$ are shown in
        Figure~\ref{fig:canonical_c6} where the colourings
        $(iii.1)-(iii.4)$ correspond to lexicographic colourings with
        respect to different vertex orders. Note that certain distinct
        vertex orderings  can give the same lexicographic colour
        pattern.

	The canonical colourings are candidates for ``unavoidable''  a collection of colourings of $H$,   one of which necessarily occurs when colouring certain host graphs $G$. Note that all three types of canonical colouring are necessary in such an unavoidable collection. Indeed, by colouring a host graph $G$ according to a canonical pattern (that is, one of the patterns \ref{item:mono}, \ref{item:rbow} or \ref{item:lex} in Definition \ref{def:canonical copy}), one can see that any copy of $H$ in $G$ will be coloured according to the same colour pattern, and so this pattern cannot be omitted from our desired unavoidable collection. As we will see, the canonical colourings of $H$ are also sufficient to form an unavoidable collection of colour patterns when colouring certain host graphs $G$.

	\begin{definition}[The canonical Ramsey
          property] \label{def:canonical ramsey property} For graphs
          $G$ and $H$, we say $G$ has the \emph{$H$-canonical Ramsey
            property}, and write $G\canarrow H$, if for every edge
          colouring $\chi:E(G)\rightarrow \bbn$ and every
          ordering~$\sigma$ of~$V(H)$, there is a copy of~$H$ in~$G$
          which is canonical with respect to~$\sigma$ when coloured
          by~$\chi$.
	\end{definition}

	\begin{figure}[h]
		\centering
		\captionsetup{width=.92\linewidth}
		\includegraphics[scale=0.9]{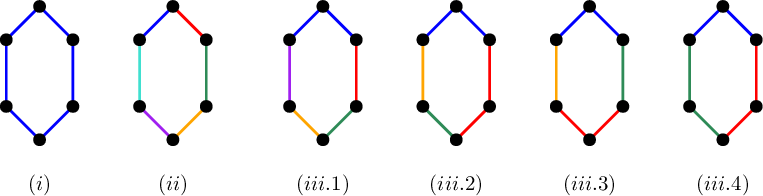}
		\caption{   \label{fig:canonical_c6} The canonical colourings of $C_6$.} 
	\end{figure}
	
	We remark that our definition for the canonical Ramsey property requires that we find either a monochromatic copy of $H$, a rainbow copy of $H$ or \emph{all possible} lexicographic copies of $H$. For example, if $G\canarrow C_6$, then  if we do not find copies of $C_6$ with colour patterns $(i)$ or $(ii)$ in a colouring of $G$ then we will find copies of $C_6$ with \emph{all} the colour patterns $(iii.1)-(iii.4)$  in Figure \ref{fig:canonical_c6}. This is the strongest possible notion of a canonical Ramsey property for colour patterns that one could ask for, as discussed in more detail after the statement of Theorem \ref{thm:canon RG even-cycles}.
	
	The following is a direct consequence of the celebrated  \emph{canonical Ramsey theorem}, proved by Erd\H{o}s and Rado~\cite{ER50} in 1950.
	
	\begin{theorem}[\!\!\cite{ER50}] \label{thm:ERcanRam}
		For a graph $H$, one has that $K_n\canarrow H$ for all sufficiently large $n\in \NN$. 
	\end{theorem}
	
	Erd\H{o}s and Rado worked only on the case when $H=K_m$ for
        some $m\in \NN$ and they also proved analogues for hypergraph
        cliques. In the case that $H=K_m$, there is just one
        lexicographic colour pattern possible.  Moreover,  
	we note that a  copy of $K_{v(H)}$ that is lexicographic necessarily contains all possible lexicographic copies of $H$. This observation along with the fact that one can find monochromatic/rainbow copies of $H$ in monochromatic/rainbow copies of $K_{v(H)}$ gives that Theorem \ref{thm:ERcanRam} does indeed follow directly from the canonical Ramsey theorem of Erd\H{o}s and Rado \cite{ER50}. We also remark that
	Erd\H{o}s and Rado actually worked in a different setting focusing on ordered embeddings of $K_m$ instead of colour patterns (which nonetheless implies Theorem \ref{thm:ERcanRam} for $H=K_m$). Following \cite{AJ05,JW04}, we  choose instead to work with colour patterns and find this approach more natural when considering $H$ that are not complete. We refer to~\cite[Section 1.4]{AKMM-1} for a more detailed discussion of alternative characterisations of the canonical Ramsey property.

	Theorem \ref{thm:ERcanRam} built on the classical Ramsey
        setting where only a bounded number of colours is used. We say that
        $G$ has the \emph{$r$-Ramsey property with respect to~$H$}, denoted $G\to(H)_r$, if any
        $r$-colouring of the edge set of $G$ results in a
        monochromatic copy of $H$. Ramsey's theorem \cite{R30}, from
        which Ramsey theory inherits its name, states that for any $H$
        and $r\in \NN$, if $n$ is sufficiently large, then
        $K_n\to(H)_r$. Thus Theorem \ref{thm:ERcanRam} showed that,
        although one can no longer guarantee monochromatic copies of
        $H$ when the size of the colour palette is allowed to grow
        with $n$, we can still have a strong Ramsey-type statement
        guaranteeing one of a small collection of unavoidable colour
        patterns.

	\subsection{Ramsey properties of random graphs} \label{sec:Ramsey random} Extending Ramsey's theorem in an orthogonal direction, a prominent theme  has been to determine the \emph{threshold} for Ramsey properties in random graphs. Here and throughout, we denote by $\bfg(n,p)$ the binomial random graph with edge probability $p=p(n)$ and  we say that a function
	$\hat p = \hat p(n)$ is the \emph{threshold} for a monotone increasing
	graph property $\cP$ if
	\[
	\lim_{n\to\infty}\Pr(\bfg(n,p)\text{ satisfies }\cP) =
	\begin{cases}0&\text{if }p =o( \hat p), \\
		1&\text{if }p =\omega( \hat p).
	\end{cases}
	\]
	We refer to $\hat p$ as \textit{the} threshold for $\cP$, although it
	is not uniquely defined but rather defined up to a choice of constant.

	Initial interest in studying Ramsey thresholds came from the
        desire to obtain \emph{sparse} Ramsey graphs. Indeed, Ramsey's
        theorem gives that $K_n\to(H)_r$ for all large enough
        $n\in \NN$ but in fact one can replace $K_n$ by much sparser
        $n$-vertex graphs $G$. In fact, as shown by R\"odl and
        Ruci\'nski in a seminal series of papers~\cite{RR93,RR94,RR95}
        in the 90s, almost all graphs have the Ramsey property and
        this remains true when focusing on much sparser graphs. For
        all $H$ and $r\in \NN$, they completely determined the
        threshold for the $r$-Ramsey property, showing that it is
        governed by the parameter~$m_2(H)$ of $H$ known as its
        \emph{maximum $2$-density}.
	
	\begin{theorem}[Rödl--Ruci\'nski~\cite{RR93,RR94,RR95}]
		\label{RR-original}
		Let $r\geq 2$ be an integer and $H$ be a  graph with at least two edges which is
		not a star forest. Then $n^{-1/m_2(H)}$ is the threshold for the
		property $\bfg(n,p)\to(H)_r$, where $m_2(H)$ is the maximum $2$-density of $H$, defined by \[
		m_2(H):=\max\left\{\tfrac{e(F)-1}{v(F)-2}:F\subseteq H,\, v(F)>2
		\right\}.
		\]
	\end{theorem}
	
	The work of R\"odl and Ruci\'nski was highly celebrated with
        both the $0$-statement and $1$-statement needing intricate
        proofs.  More recently, Nenadov, Person, \v{S}kori\'{c} and
        Steger~\cite{NPSS17} obtained several related
        $0$-statement results, and Nenadov and Steger~\cite{NS16}
        showed how to reprove the $1$-statement using the theory of
        hypergraph containers~\cite{BMS14,ST15}, leading to an elegant
        and concise proof.
	
	\subsection{Canonical Ramsey property for even cycles in
          random graphs} \label{sec:our results} With regards to
        canonical Ramsey properties of random graphs, until recently
        very little was known. Due to Theorem \ref{RR-original}, if
        $H$ is a graph which is not a star forest and
        $p=o(n^{-1/m_2(H)})$, then asymptotically almost surely
        (a.a.s.\ for short) $\bfg\sim \bfg(n,p)$ can be $2$-coloured,
        by a colouring $\chi$ say, avoiding monochromatic copies of
        $H$. For all such $H\neq K_3$, there is some ordering $\sigma$
        of $V(H)$ such that the lexicographic colouring of $H$ with
        respect to $\sigma$ uses more than 2 colours. Hence the
        colouring $\chi$ of $\bfg$ avoids at least one lexicographic
        pattern as well as both the monochromatic and rainbow
        patterns. This thus gives that the threshold for the property
        $\bfg(n,p) \canarrow H$ is at least $n^{-1/m_2(H)}$.
	
	Our main theorem shows that this lower bound is tight up to a
        logarithmic factor in the case that $H=C_{2k}$ is an even
        cycle, noting that $m_2(C_{2k})^{-1}=1 - 1/(2k-1)$ for $2\leq
        k\in \NN$.
	
	\begin{theorem}\label{thm:canon RG even-cycles}
		Let $k\geq 2$ be an integer. 
		If $p= \omega\left(
		n^{-1+{1}/{(2k-1)}} \, \log{n}\right)$,  then a.a.s.\
		\[
		\bfg(n,p) \canarrow C_{2k}.
		\]
	\end{theorem}
	
	\begin{remark}
          We remark that with more care we could prove a slightly
          stronger statement, requiring only
          $p\geq Cn^{-1+{1}/{(2k-1)}} \, \log{n}$ for some large
          enough constant $C>0$, and guaranteeing not just one
          canonical copy of $C_{2k}$ but $\Omega(n^{2k}p^{2k})$ many
          copies. We opt to prove the form of Theorem~\ref{thm:canon
            RG even-cycles} given here to simplify the exposition.
          Another reason is that we do not believe that the~$\log n$
          factor that we currently have is in fact required (see
          Section~\ref{sec:related}).
	\end{remark}
	
	In defining our canonical Ramsey property in Definition \ref{def:canonical ramsey property}, we required that all lexicographic copies of $H$ are present when the colouring avoids monochromatic and rainbow copies of $H$. At the other end of the spectrum, one could alternatively require just one lexicographic copy of $H$.


	\begin{definition}[The weak canonical Ramsey property] \label{def:weak canonical ramsey property} For graphs $H$ and $G$, we say $G$ has the weak $H$-canonical Ramsey property, and write $G\wcanarrow H$, if for every edge colouring $\chi:E(G)\rightarrow \bbn$ there is a copy of $H$ in $G$ coloured by $\chi$ which is canonical with respect to \emph{some} ordering $\sigma$ of $V(H)$. 
	\end{definition}
	
	When $k\geq 3$, any lexicographic colouring of $C_{2k}$ uses at least 3 colours and so by the same reasoning as before, $n^{-1/m_2(C_{2k})}$ is also a lower bound for the weak canonical Ramsey property. 
	However for cycles of length four, we see a change in behaviour and have that the threshold for $\bfg(n,p)\wcanarrow C_4$ is in fact  much smaller than $n^{-2/3}=n^{-1/m_2(C_4)}$.
	
	\begin{proposition} \label{obs:weak C4}
          The threshold for $\bfg(n,p)\wcanarrow C_4$ is $n^{-3/4}$.
        \end{proposition}

        The proof of Proposition~\ref{obs:weak C4} is given in Section~\ref{sec:prop_weak_C4}.
	
	\subsection{Related works}
        \label{sec:related}
        In general, we believe that
        $n^{-1/m_2(H)}$ is the threshold for the canonical Ramsey
        property with respect to $H$ for any connected graph~$H$ that is not a star forest or a $K_3$. In accompanying work~\cite{AKMM-1,alvarado26:_ramsey}, we provide evidence for this by showing that when $p=\omega(n^{-1/m_2(H)})$, a.a.s.\ one can find canonical copies of $H$ in any colouring of $\bfg(n,p)$ that respects some prefixed bounded lists of colours assigned to each edge. In fact this result  in the case that $H=C_{2k}$,  or rather a slightly stronger version applied to random subgraphs of uniformly dense graphs (see Theorem~\ref{thm:can-list-RG}),  is also a key component of our proof here. 
	
	Moreover, in simultaneous and independent work, Kam\v{c}ev and 
        Schacht~\cite{kamcev23:_canon,kamcev25:_canon_JLMS} obtained a remarkable result proving that
        $n^{-1/m_2(K_m)}$ is indeed the threshold for
        $\bfg(n,p)\canarrow K_m$ for $m\geq 4$ and thus completely
        resolved the problem for the case when $H$ is a clique. This
        beautiful result appeals to the transference principle of
        Conlon and Gowers~\cite{CG16}, which is a very general result
        obtained by
        using an analytic approach. This principle pre-dates the theory of  hypergraph containers~\cite{BMS14,ST15} and can be used to get many of the same  consequences in random graphs, in the case that the subgraph $H$ of interest is \emph{strictly balanced}, that is, when the maximum in the definition of $m_2(H)$ is achieved by $F=H$ only. For instance, the transference principle was used to recover the $1$-statement of Theorem \ref{RR-original} for strictly balanced $H$. 
	Kam\v{c}ev and Schacht~\cite{kamcev23:_canon,kamcev25:_canon_JLMS} use and adjust this principle to achieve their result for cliques. As a consequence, their methods imply results for all strictly balanced $H$, including even cycles, giving canonical copies of $H$ in any colouring of $\bfg(n,p)$ a.a.s.\ when $p=\omega(n^{-1/m_2(H)})$. In particular this gives an upper bound on the threshold  for the property that $\bfg(n,p)\wcanarrow H$ and determines this threshold for all strictly balanced $H$ other than some small examples (as in Proposition~\ref{obs:weak C4}). Their proof does not imply, however, that one can find \emph{all} lexicographic copies of $H$ and so does not directly imply Theorem \ref{thm:canon RG even-cycles}. 

\subsection{Proof components} Our proof 
incorporates several other related problems for subgraphs in random
graphs. Indeed, a key component of our proof is to use a canonical
Ramsey theorem with list constraints  which we previously
developed~\cite{AKMM-1,alvarado26:_ramsey} and  proved using the hypergraph container
method. We also use relative Tur\'an results for cycles in random
graphs (see Theorem~\ref{thm:turan}). Finally, we appeal to a recent
result on finding directed cycles  in orientations of $\bfg(n,p)$ (see
Lemma \ref{lem:orient}), and use this to find \emph{all} lexicographic
copies of $C_{2k}$ under certain conditions of our colouring of
$\bfg(n,p)$. We believe that this somewhat surprising connection to a
Ramsey problem involving orientations of graphs is one of the most
interesting features of our proof. 
        
\subsubsection*{Organisation} In Section~\ref{sec:prelims} we collect
some notation, basic properties of random graphs and some theory on
locally dense graphs.  In Section~\ref{sec:prop_weak_C4}, we give the
proof of Proposition~\ref{obs:weak C4}.  In Section~\ref{sec:proof
  overview}, we give a proof overview and reduce
Theorem~\ref{thm:canon RG even-cycles} to a key lemma, namely
Lemma~\ref{lem:multicolour-paths}, which proves the existence of a
locally dense ``rainbow focused'' graph generated by $\bfg(n,p)$ (see
Section~\ref{sec:proof overview} for the relevant definitions). In
Section \ref{sec:cycles} we then collect some results on the
distribution of even cycles in $\bfg(n,p)$ that we will need. In 
Section \ref{sec:rainbow paths} we prove Lemma
\ref{lem:multicolour-paths} and hence complete the proof of Theorem
\ref{thm:canon RG even-cycles}. 

\subsubsection*{Acknowledgements} We are grateful to Nina Kam\v{c}ev and Mathias Schacht for insightful conversations about this topic and the anonymous referees for their careful work.

\section{Preliminaries} \label{sec:prelims} In this section, we
collect the necessary theory and tools that we will use in our proof.
These include simple random graph properties in Section
\ref{sec:randomgraphs}, properties concerning the distribution of
paths in random graphs in Section \ref{sec:paths} and the theory of
locally dense graphs in Section \ref{sec:localdense}. Before all of
this though, in Section \ref{sec:notation}, we collect the relevant
notation.
\subsection{Notation} \label{sec:notation} We write
$P_\ell$ for the path with $\ell$ vertices (and hence $\ell-1$ edges)
and~$C_\ell$ for the cycle with~$\ell$ vertices.  For a graph $G$, a
vertex $v\in V(G)$ and a vertex subset 
$U\subseteq V(G)$, we have that $d_G(v,U)$ denotes the number of
neighbours of $v$ in $U$ and in the case that $U=V(G)$, we use the
shorthand $d_G(v)$ to denote the degree of $v$ in $G$.  We have that
$G[U]$ denotes the graph induced by $G$ on the vertex set $U$ and
$e_G(U)=e(G[U])$ denotes the number of edges in $G$ that lie in
$U$. Moreover, for disjoint vertex subsets $U,W\subseteq V(G)$,
$e_G(U,W)$ counts the number of edges with one endpoint in $U$ and the
other in
$W$. 
Given a fixed (non-empty) graph $H$, we say that a graph $G$ is
$H$-free if $G$ contains no copy of $H$. Let $\ex(n,H) := \max\{ e(G)
: G \text{ has $n$ vertices and is } H\text{-free} \}$. The {\em Turán
  density} $\pi(H)$ 
{\em of} $H$ is defined as
\begin{equation} \label{eq:turan_density_def}
  \pi(H) :=
  \lim_{n\to\infty}\ex(n,H)\binom{n}{2}^{-1}.
\end{equation}	 
The well-known Erd\H{o}s--Stone theorem implies that $\pi(H) =
(r-2)/(r-1)$ for every graph~$H$ with chromatic number $r\geq2$.
	
The binomial random graph, denoted $\bfg(n,p)$, refers to the
probability distribution of graphs on vertex set $[n]$ obtained by
taking every possible edge independently with probability $p=p(n)$. We
say an event happens asymptotically almost surely (a.a.s.\ for short)
in $\bfg \sim \bfg(n,p)$ if the probability that it happens tends to 1
as $n$ tends to infinity. We will also use standard asymptotic
notation throughout, with asymptotics always being taken as the number
of vertices $n$ tends to infinity. For a Boolean statement $A$, we
denote by $\mathds{1}[A]$, the indicator function which evaluates to
$1$ if $A$ holds and $0$ if $A$ does not hold.

Finally, we use the notation $a=(b\pm c)d$ to denote that $a\leq
(b+c)d$ and $a\geq (b-c)d$ and we omit floors and ceilings throughout,
so as not to clutter the arguments.
	
\subsection{Random graph properties} \label{sec:randomgraphs}
	
Here we collect some simple properties of random graphs.
\begin{lemma} \label{lem:RG - properties} For any $0<\eps\leq 1$ and
  $p=p(n)$ such that $p=\omega(n^{-1})$, the following properties hold
  a.a.s.\ in $\bfg \sim \bfg (n,p)${\rm:}
  \begin{enumerate}[{\rm(1)}]
  \item \label{RG: edge count} $e_\bfg (U)=(1\pm \eps) p |U|^2/2$ for
    all $U\subseteq V(G)$ with $|U|\geq \sqrt{\tfrac{12n}{\eps^2 p}}$.
  \item \label{RG: degrees} $d_\bfg(v)=(1 \pm \eps)pn$ for all
    vertices $v\in V(\bfg)$.
  \item \label{RG:edges between subsets} $e_\bfg(U,W)=(1\pm
    \eps)p|U||W|$ for all pairs of disjoint subsets $U, W\subseteq
    V(\bfg)$ such that $|U||W|\geq \tfrac{6n}{\eps^2 p}$.
  \item \label{RG:bad vts} For every $U\subseteq V(G)$ with $|U|\geq
    \eps n$, all but at most $\eps^2 n$ vertices $v\in V(G)$ have
    $d_\bfg(v,U)=(1\pm \eps)p|U|$.
  \end{enumerate}
\end{lemma}
\begin{proof}
  The first three assertions follow from standard applications of
  Chernoff bounds (see for example \cite[Theorem 2.1]{JLR00}) and
  union
  bounds. 
	
	
  In order to prove~\ref{RG:bad vts}, we can assume $\eps\leq 1$ as
  otherwise there is nothing to prove.  We also assume that Lemma~\ref{lem:RG -
    properties}\ref{RG: edge count}--\ref{RG:edges between subsets}
  a.a.s.\ hold in
  $\bfg \sim \bfg(n,p)$.  Thus, for every pair of disjoint vertex
  subsets $U$ and $W$ with $|U||W|\geq \tfrac{24 n}{\eps^2p}$, one has
  that $ e_\bfg(U,W) = \left(1\pm \tfrac{\eps}{2}\right) p|U||W|.$
  Secondly, we have that $e_\bfg(U)=\left(1\pm\tfrac{\eps
    }{2}\right)p|U|^2/2$ for any set $U\subseteq V(G)$ with $|U|\geq
  \sqrt{\tfrac{48n}{\eps^2p}}$. Letting $G$ be any outcome of $\bfg$
  satisfying these two events on a sufficiently large number of
  vertices $n$, we will show that the desired conclusion of the lemma
  holds in $G$.  So fix some $U$ such that $|U|\geq \eps
  n$. Furthermore, let $B^-$ be the set of vertices $v$ such that
  $d_G(v,U)<(1-\eps)p|U|$ and $B^+$ be the set of vertices $v$ in $G$
  such that $d_G(v,U)>(1+\eps)p|U|$. Supposing for a contradiction that
  $|B^-| > \tfrac{\eps^2n}{2}$, let $B\subseteq B^-$ be a set of size
  exactly $\tfrac{\eps^2n}{2}$ and note that $|U\setminus B|\geq
  |U|-\tfrac{\eps^2 n}{2}\geq \left(1-\tfrac{\eps}{2}\right)|U|$,
  using that $|U|\geq \eps n$.  Therefore, since $p=\omega(n^{-1})$
  and $n$ is sufficiently large, we have that $|U\setminus B||B|\geq
  \tfrac{24 n}{\eps^2p}$ and counting edges from $B$ we get that
  \[e_G(B,U\setminus B)<|B| \cdot (1-\eps)p|U|\leq
    \left(1-\frac{\eps}{2}\right)p|B||U\setminus B|, \]
  contradicting our first assumption on $G$. The case for vertices of
  too high degree is similar. We suppose for a contradiction that
  $|B^+|\geq \tfrac{\eps^2n}{2}$ and let $B'\subseteq B^+$ be a subset
  of size exactly $\tfrac{\eps^2n}{2}$. Now we claim that $e_G(B')\leq
  \tfrac{\eps}{2} p|B'||U|$. Indeed $\tfrac{\eps^2n}{2}=|B'|\geq
  \sqrt{\tfrac{48n}{\eps^2p}}$ as $p=\omega(n^{-1})$ and so our second
  assumption on $G$ applied to $B'$ implies that $e_G(B')\leq
  p|B'|^2\leq \tfrac{\eps^2}{2} pn|B'|\leq \tfrac{\eps}{2}p|B'|U|$ as
  $|U|\geq \eps n$.  Therefore, counting edges from $B'$, we have that
  \[e_G(B',U\setminus B')>|B'| \cdot (1+\eps)p|U|-e_G(B')\geq
    \left(1+\frac{\eps}{2} \right)p|B'||U|\geq \left(1+\frac{\eps}{2}
    \right)p|B'||U\setminus B'|, \]
  again contradicting our first condition of $G$.
\end{proof}
	
	
\subsection{Paths in random graphs} \label{sec:paths}
	
In this section, we explore the distribution of paths in
$\bfg(n,p)$. We make the following definitions.
	
\begin{definition} \label{def:path count}

  For a graph $G$ on vertex set $[n]$, $\ell\in \bbn$ and a pair
  $\{u,v\}$ of vertices, we let $X^\ell_G(\{u,v\})$ denote the number
  of paths with $\ell$ vertices in $G$ with endpoints $u$ and $v$.
  Moreover, for a vertex $v\in [n]$, we let
  \[X^\ell_G(v):=\sum_{u\in [n]\setminus \{v\}}X^\ell_{G}(\{u,v\})\]
  be the number of copies of $P_\ell$ in $G$ that have $v$ as an
  endpoint.
\end{definition}
	
	
	
We need the following statement which will be used to show that if we
find a collection of $\Omega(n^\ell p^{\ell-1})$ copies of $P_\ell$ in
$\bfg(n,p)$, then there will be $\Omega(n^2)$ pairs of vertices $f\in
\binom{[n]}{2}$ that form the endpoints of some path in the
collection.
	
\begin{lemma} \label{lem:well-distributed}
  For any integer $\ell>2$, $\xi >0$ and $p=p(n)$ such that $p
  =\omega\left( n^{-1+1/(\ell-1)}\right) $, the following holds
  a.a.s.\ in $\bfg\sim \bfg(n,p)$. Let $F=F(\bfg)$ be the collection
  of pairs $f\in \binom{[n]}{2}$ such that
  $X_\bfg^\ell(f)>2n^{\ell-2}p^{\ell-1}$. Then
  \[
    \sum_{\substack{f\in F}}{X^\ell_{\bfg}(f)} \leq \xi n^{\ell}
    p^{\ell-1}.
  \]
\end{lemma}
\begin{proof}
  Fix an integer $\ell \geq 2$ and $\xi>0$ and let $\bfg\sim
  \bfg(n,p)$ with $p= \omega\left(n^{-1+\frac{1}{\ell-1}}\right)$.  If
  we consider $f$ as a uniform random pair in $E(K_n)=\binom{[n]}{2}$
  (selected independently of $\bfg$), then due to the fact that
  $\binom{n}{2}\leq n^2$, one can check that 
  our desired conclusion will follow from
  proving the statement that
  \begin{linenomath} \begin{align}\label{eq : probab-well-distrib}
      \Pr_p\left[ \Ex_f\left[ X_\bfg^\ell(f) \mathds{1}{ [
            X_\bfg^\ell(f) > 2 n^{\ell-2}p^{\ell-1} ] } \right] >
        \xi n^{\ell-2}p^{\ell-1} \right] = o(1),
  \end{align} \end{linenomath}
where $\Pr_p := \Pr_{\bfg\sim \bfg(n,p)}$ and $\Ex_f := \Ex_{f\sim {\rm Unif}(E(K_n))}$. In order to show \eqref{eq : probab-well-distrib}, we use Markov's Inequality, the Fubini--Tonelli Theorem and the  Cauchy--Schwarz Inequality to obtain that
\begin{linenomath}
  \begin{align}
    \nonumber \Pr_p & \left[ \Ex_f\left[
                      X_\bfg^\ell(f) \mathds{1}{[ X_\bfg^\ell(f) \geq 2
                      n^{\ell-2}p^{\ell-1} ] } \right] >\xi \,
                      n^{\ell-2}p^{\ell-1}\right] \\
    \nonumber & \leq \frac{\Ex_p
                \Ex_f \left[ X_\bfg^\ell(f) \mathds{1}{[ X_\bfg^\ell(f) \geq 2
                n^{\ell-2}p^{\ell-1} ] } \right] } { \xi n^{\ell-2}
                p^{\ell-1}} \\
    \label{eq:prob esitimate} & \leq \frac{\Ex_f \left[ \left(\Ex_p [ X_\bfg^\ell(f)^2 ]\right)^{1/2} \, \left(\Pr_p [{  X_\bfg^\ell(f) \geq 2 n^{\ell-2}p^{\ell-1}} ]\right)^{1/2} \right]  }{  \xi n^{\ell-2} p^{\ell-1}}.
\end{align}
\end{linenomath}
We claim then that it is enough to show that for each (deterministic)
$f\in E(K_n)$, one has
\begin{linenomath} \begin{align}\label{eq:second-moment}
  \Ex_p[ X_\bfg^\ell(f)^2 ] = (1+o(1)) \mu^2,
\end{align}  \end{linenomath}
where $\mu:=\Ex_p[X_\bfg^\ell(f)]=(n-2)_{(\ell-2)} \, p^{\ell-1}$ is the expected number of copies of $P_\ell$ that have endpoints $f$. 
Indeed, assuming~\eqref{eq:second-moment}, by the second moment method we have that 
\begin{multline*}
  \qquad\qquad
  \Pr_p [{ X_\bfg^\ell(f) \geq 2 n^{\ell-2}p^{\ell-1} } ] \leq \Pr_p
  [{ X_\bfg^\ell(f) \geq 2 \mu} ]\\ \leq \Var_p[X^\ell_\bfg(f)] /\mu^2 =
  \big[\Ex_p[ X_\bfg^\ell(f)^2 ]-\mu^2\big]/\mu^2=o(1),\qquad\qquad
\end{multline*}
and combining this with \eqref{eq:prob esitimate}, we have that
\[
  \Pr_p\left[ \Ex_f\left[ X_\bfg^\ell(f) \mathds{1}{ [ X_\bfg^\ell(f)
        > 2 n^{\ell-2}p^{\ell-1} ] } \right] > \xi
    n^{\ell-2}p^{\ell-1} \right] = o\left( \frac{\mu }{ 
      n^{\ell-2}p^{\ell-1}} \right) = o(1),
\]
establishing \eqref{eq : probab-well-distrib}, as desired.
		
Thus it only remains to prove~\eqref{eq:second-moment}. Note that this is
equivalent to the statement that $\Var_p[X_\bfg^\ell(f)] =
o(\mu^2)$. Since $X_\bfg^\ell(f)$ is a sum of indicators random
variables, it is well-known (see for example~\cite[Chapter 4]{AS})
that $\Var_p[X_\bfg^\ell(f)]\leq\mu (1+\Delta^*) $ where $\Delta^* =
\sum_{\substack{P}}{ \Pr_p[ P \subset \bfg \mid Q \subset \bfg ] }$,
where~$Q$ is an arbitrary copy of $P_\ell$ in $K_n$ with endpoints~$f$,
and 
the sum goes over all copies~$P$ of~$P_\ell$ in~$K_n$ whose endpoints
are~$f$ and are such that $1\leq |E(P) \cap E(Q)| < \ell-1 $. By our
assumption on $p$, we have that $\mu =\omega(1)$, and so we only need
to show that $\Delta^* = o(\mu)$. Now suppose that $1\leq r \leq
\ell-2$ and $I\subseteq E(Q)$ is some subset of $r$ edges of $Q$. Then
the number of paths $P$ such that $E(P)\cap E(Q) = I$ is at most $
n^{\ell-2-v_I}$, where $v_I$ is the number of vertices in
$V(Q)\setminus f$ that belong to the edges in~$I$. Since $I$ defines a
disjoint union of paths (each path of length at least one), we have
that $v_I \geq r$. So, the total contribution to $\Delta^*$ of paths
$P$ such that $E(P)\cap E(Q)=I$ is at most $n^{\ell-2-r}p^{\ell-1-r}$.
Hence, summing all the possible contributions, we obtain that
\begin{linenomath}
  \begin{align*}
    \Delta^* & = \sum_{r=1}^{\ell-2}{
               \sum_{\substack{I\subset E(Q): |I| = r}}{ n^{\ell-2-r}
               p^{\ell-1-r} } } = \sum_{r=1}^{\ell-2}{ \binom{\ell-1}{r}{
               \frac{ n^{\ell-2}p^{\ell-1}}{(np)^r} } }\leq2^{\ell-1} \cdot
               \frac{n^{\ell-2}p^{\ell-1}}{np}=o(\mu),
  \end{align*}
\end{linenomath}
where the last equation follows because $pn =
\omega\left(n^{\frac{1}{\ell-1}}\right)=\omega(1)$, and we are
done.
\end{proof}
\subsection{Locally dense graphs} \label{sec:localdense}
	
The final tool that we will need is the notion of locally dense
graphs.
\begin{definition}[$(\rho,d$)-dense graphs]
  For $\rho,d\in (0,1]$, we say a graph $G$ is \emph{$(\rho,d$)-dense}
  if for every $S\subset V(G)$ with $|S|\geq \rho n$, the induced
  graph $G[S]$ has at least $d\binom{|S|}{2}$ edges.
\end{definition}
		
\begin{remark} \label{rem:simple local dense} In~\cite{RR95} the
  authors observed that in order to establish that a graph~$G$ is
  $(\rho,d)$-dense it suffices to establish the defining property for
  subsets $S\subset V(G)$ of size exactly~$\lceil\rho n\rceil$.
\end{remark}
	
Graphs with the $(\rho,d)$-dense property for $\rho = o(1)$ are often
called {\em locally dense} graphs. This property can be viewed as a
weak quasirandomness property.

\section{The weak canonical threshold for~$C_4$}
\label{sec:prop_weak_C4}

Here we prove Proposition~\ref{obs:weak C4}, which asserts that the
threshold for the weak canonical property $\bfg(n,p)\wcanarrow C_4$
is~$n^{-3/4}$.
	
\begin{proof}[Proof of Proposition~\ref{obs:weak C4}]
  This essentially follows from a result of Barros, Cavalar, Parczyk
  and the third author of the current paper~\cite{BCMP21}. Indeed
  in~\cite{BCMP21} it is shown that $n^{-3/4}$ is the threshold for
  $G=\bfg(n,p)$ to have the \emph{anti-Ramsey} property with respect
  to $C_4$, meaning that in any \emph{proper} colouring of $G$ there
  is a rainbow copy of $C_4$. Thus, if $p=o(n^{-3/4})$ then a.a.s.\
  there is a proper colouring of $\bfg\sim \bfg(n,p)$ avoiding rainbow
  copies of $C_4$. As the colouring is proper, it also avoids
  monochromatic copies of $C_4$ and all lexicographic copies of $C_4$
  (as for any ordering~$\sigma$ of $V(C_4)$ the first vertex in the
  ordering is incident to two edges of the same colour). This
  establishes the $0$-statement. For the $1$-statement, we also follow
  the approach of~\cite{BCMP21}, which uses that if
  $p=\omega(n^{-3/4})$, then
  $\bfg\sim\bfg(n,p)$ a.a.s.\ contains a copy of $K_{2,4}$ and
  $K_{2,4}$ has the anti-Ramsey property with respect to
  $C_4$. Therefore it suffices to consider non-proper colourings
  of~$K_{2,4}$.  Fix a non-proper colouring of $K_{2,4}$ and assume
  that there is no lexicographic copy of~$C_4$ within
  it. Our aim is to show that it contains a
    monochromatic copy of~$C_4$.
		
   Label the vertex set of $K_{2,4}$ as $X\cup Y$ with $X$ the part of size $2$ and $Y$ the part of size 4 in the bipartition.   Without loss of generality we can assume that colour $1$ induces a
    subgraph of $K_{2,4}$, say~$G_1$, such that $\Delta(G_1)\geq
    2$. Observe that if $x_1,x_2\in X$ and  $y_1,y_2\in Y$ are distinct vertices 
   with $y_1\in N_{G_1}(x_1) \cap N_{G_1}(x_2)$, then $y_2\in 
    N_{G_1}(x_1) \cup N_{G_1}(x_2)$, as if this is not the case, then it
    is easy to check that we obtain a lexicographic copy of $C_4$ in
    our colouring. Likewise if $x_1\in N_{G_1}(y_1) \cap N_{G_1}(y_2)$, then $x_2\in 
    N_{G_1}(y_1) \cup N_{G_1}(y_2)$. In what follows we conclude the proof by showing
    that $G_1$ contains a $C_4$.  Because of the observation  just
    discussed, if there is a vertex $x\in X$ with at least
    $2$~neighbours in $G_1$, 
    then there is a path in $G_1$ with three edges and in particular  there exists $y\in Y$
    with $\deg_{G_1}(y)\geq 2$. Now, using our observation once more,
    we have that each vertex in $Y\setminus \{y\}$ is adjacent to one
    of the vertices in~$X$. By the pigeonhole principle, there is a
    vertex $x\in X$ which is adjacent to two vertices of $Y\setminus
    \{y\}$, say $y'$ and $y''$. Again by our observation, the
    remaining vertex $x'$ in $X\setminus \{x\}$ is adjacent to either $y'$ or $y''$,
    completing a~$C_4$ in~$G_1$ in either case, as desired.
\end{proof}
	
\section{Proof overview} \label{sec:proof overview}
	
Our proof uses ideas from both the original proof of R\"odl and
Ruci\'nski~\cite{RR95} and the approach of Nenadov and Steger
\cite{NS16} proving the $1$-statement of the sparse Ramsey theorem
with a finite number of colours. Let us sketch briefly both ideas, 
restricting our discussion to our case of interest,
where the desired monochromatic graph $H$ is an even cycle, say
$H=C_{2k}$.
	
Firstly, the proof of R\"odl and Ruci\'nski~\cite{RR95} uses the idea
of \emph{colour-focusing}, a classical approach in Ramsey
theory, along with \emph{multi-round exposure} techniques. Suppose
first that we are interested in the two colour case and we want to
show that any red/blue-colouring of $\bfg(n,p)$ with
$p=\omega\left(n^{-1+1/(2k-1)}\right)$ has a monochromatic copy
of~$C_{2k}$ a.a.s. They then reveal $\bfg(n,p)$ in two rounds. In the first
round, they reveal $\bfg_1=\bfg(n,p_1)$, where $p_1=\alpha p$ for some
small $\alpha>0$, and so a positive fraction of the edges of
$\bfg(n,p)$ has been revealed. Now for any red/blue-colouring $\chi_1$
of $E(\bfg_1)$, they prove that there are $\Omega(n^{2k}p^{2k-1})$
monochromatic copies of $P_{2k}$. The key point is that for each such
$(u,v)$-path~$P_{2k}$ which is monochromatic, say, blue, if $uv$ appears in
the second exposure, then our hand is forced if we want to colour $uv$
avoiding a monochromatic copy of $C_{2k}$: we \emph{have} to colour it
red. Now using that there are many monochromatic copies of $P_{2k}$,
they show that there is some colour, say blue, and some locally dense
graph $\Gamma$ on $\Omega(n)$ vertices of $\bfg_1$ such that for each
edge $f=uv$ of $\Gamma$, there is some $(u,v)$-path with $2k$ vertices
in $\bfg_1$ which is coloured completely blue by $\chi_1$. Hence
revealing now $\bfg_2=\bfg(n,p_2)$ such that $\bfg(n,p)=\bfg_1\cup
\bfg_2$, we have that if $\bfg_1$ is coloured by $\chi_1$, any edge in
$E(\bfg_2)\cap \Gamma$ must be coloured red if we want to avoid blue
copies of $C_{2k}$ in $\bfg$. The final step is to show that a.a.s.\
(in fact, with probability large enough to perform a union bound over
all possible colourings of $\bfg_1$), we have that $\bfg_2\cap \Gamma$
contains a copy of $C_{2k}$, which is coloured completely red, and
hence the colouring $\chi_1$ cannot be extended to the edges of
$\bfg_2$ without creating a monochromatic cycle. The proof for more
colours is achieved by a careful induction on the number of colours,
iterating these ideas.
	
The proof of Nenadov and Steger \cite{NS16} is entirely distinct and
appeals to the method of hypergraph containers. Here, suppose we want
to find a monochromatic $C_{2k}$ in any $r$-colouring of a typical
$\bfg=\bfg(n,p)$. The idea is then to create an auxiliary hypergraph
$\cH$ whose vertex set consists of $r$ copies of $E(K_n)$ and whose
edge set encodes the monochromatic copies of $C_{2k}$ in $K_n$. The
key observation then is that any colouring of $\bfg$ that avoids
monochromatic copies of~$C_{2k}$ can be identified with an
independent set in~$\cH$. Using containers, one can efficiently group
together these independent sets and identify a small set of
\emph{containers} $\cC$ such that each independent set of $\cH$
belongs to one such container. The proof then proceeds by showing that
for each \emph{fixed} container $C\in \cC$, it is very unlikely (that
is, with probability $\exp(-\Omega(n^2p))$) that the graph~$\bfg$ lies
within $C$. By this, we mean that there is a colouring of $\bfg$ that,
when mapped in the obvious way to a vertex subset of $\cH$,
corresponds to a subset of $C$. The proof then follows by performing a
union bound over the choices of container $C\in\cC$.

Neither approach generalises directly to the setting where the number
of colours can grow with $n$. Indeed, for the first approach, even if
one could find many monochromatic copies of~$P_{2k}$ in each random
graph $\bfg_i=\bfg(n,p_i)$ (which may not be the case), one would need
$t=\omega(1)$ iterations in order to focus the colours enough to force
a monochromatic copy of $C_{2k}$ in the final random graph $\bfg_t$,
thus forcing each $p_i$ to be of order $p/t$. Likewise, the containers
approach relies crucially on the fact that the hypergraph $\cH$ that
encodes the monochromatic copies of $C_{2k}$ has
$v(\cH)=O(n^2)$. Having an unbounded number of copies of $E(K_n)$ in
$V(\cH)$ corresponding to an unbounded number of colours available,
leads to an adjustment which renders the upper bound on the number of
containers useless and the proof falls apart.
	
The key idea of our proof is to use the first exposure to follow the
idea of \emph{colour-focusing} used by R\"odl and
Ruci\'nski~\cite{RR95}, but instead of finding monochromatic paths to
globally reduce the number of colours under consideration, we find
\emph{rainbow} paths in order to give \emph{local} restrictions on the
colours that can be used in the second exposure. We make the following
definition.
	
\begin{definition}[Rainbow focused graph]
  Let $\ell \geq 3$ be an integer and let $G$ be a graph whose edges
  are coloured with a colouring $\chi:E(G)\rightarrow \bbn$. Define
  $\Gamma = \Gamma(\ell,G,\chi)$ as the graph with vertex set $V(G)$
  such that $uv$ is an edge of $\Gamma$ if there is a $(u,v)$-path $P$
  with $\ell$ vertices in $G$ which is coloured rainbow by $\chi$.
\end{definition}
	
The heart of our argument is the following lemma, which will be used
on the first exposure, and shows that if a colouring $\chi$ of the
random graph $\bfg=\bfg(n,p)$ avoids monochromatic and lexicographic
copies of $C_{2k}$, then there will be many rainbow copies of $P_{2k}$
and in fact the rainbow focused graph $\Gamma(2k,\bfg,\chi)$ will be
locally dense.
	
\begin{lemma}[Locally dense rainbow focused
  graph]\label{lem:multicolour-paths} Given an integer $k \geq 2$
  there exists $d>0$ such that for all $0<\rho\leq 1$, the following
  holds a.a.s.\ in $\bfg\sim\bfg(n,p)$ when
  $p=\omega\left(n^{-1+1/(2k-1)}\right)$.  In every
  edge colouring $\chi:E(\bfg)\rightarrow \bbn$ of $\bfg$, either
  \begin{enumerate}[{\rm(i)}]
  \item \label{item:mono cycle}$\bfg$ contains a monochromatic copy
    of~$C_{2k}$, 
  \item \label{item:lexico} for all orderings $\sigma$ of $V(C_{2k})$,
    the graph~$\bfg$ contains a copy of~$C_{2k}$ which is lexicographic with
    respect to~$\sigma$, or
  \item \label{item:rainbow} the rainbow focused graph $\Gamma =
    \Gamma(2k, \bfg, \chi)$ is $(\rho,d)$-dense.
  \end{enumerate}
\end{lemma}
	
Lemma \ref{lem:multicolour-paths} will be proved in Section
\ref{sec:rainbow paths}. We remark that a relative Tur\'an result for
even cycles (see Theorem \ref{thm:turan}) is used to show the
existence of a monochromatic cycle if some colour class of~$\chi$ is dense. This
is where our proof fails for odd cycles as they are non-degenerate.

Lemma \ref{lem:multicolour-paths} will be used on our first exposure
$\bfg_1$. Note that if $\bfg_1$ is coloured by $\chi$ and
$\Gamma=\Gamma(2k,\bfg_1,\chi)$, then for every edge $f=uv\in
E(\Gamma)$, if $f$ appears in our second exposure and we wish to
colour it avoiding a canonical copy of $C_{2k}$, then we have
restrictions on what colour we use on $f$. Indeed, as $f=uv\in
E(\Gamma)$, there is a rainbow copy of $P_{2k}$ with endpoints~$u$
and~$v$. If $f$ is coloured in any colour other than the colours that
feature on this path, we will get a rainbow cycle. Therefore, for each
edge~$f$ of~$\Gamma$ there is a list of $2k-1$ colours from which we have
to choose the colour of~$f$ if~$f$ appears
in~$\bfg_2$.
The second main intermediate result we need shows that these
local restrictions are enough to force a canonical copy of $C_{2k}$ in
the random graph. In fact, this part of our argument holds much more
generally and so we state it for all possible $H$.
	
\begin{theorem}[Canonical Ramsey theorem with list
  colourings] \label{thm:can-list-RG} Let $H$ be a graph, and let $1 \leq r\in
  \bbn$ and $d>0$ be given.  Then there exists $\rho$, $c>0$ such that the
  following holds.  Let~$\Gamma$ be an $n$-vertex $(\rho,d)$-dense graph
  and let $\call:E(\Gamma)\rightarrow \mathbb{N}^r$ be a
  list-assignment (with repeats allowed). If $p = \omega\left(
    n^{-1/m_2(H)} \right)$, then with probability at least $1-
  \exp{(-cp n^2)}$, any edge colouring~$\chi$ of $\Gamma \cap
  \bfg(n,p)$ such that $\chi(e)\in \cL(e)$ for all $e\in E(\Gamma \cap
  \bfg(n,p))$ contains a canonical copy of $H$ with respect to
  $\sigma$ for all orderings $\sigma$ of $V(H)$.
\end{theorem}
	
Theorem~\ref{thm:can-list-RG} is proved in a companion
paper~\cite{AKMM-1,alvarado26:_ramsey} by building on the containers approach of Nenadov
and Steger to prove the sparse Ramsey theorem with finitely many
colours.
	
With Lemma~\ref{lem:multicolour-paths} and Theorem~\ref{thm:can-list-RG}, the proof of Theorem \ref{thm:canon RG
  even-cycles} follows readily, as we now show.
	
\begin{proof}[Proof of Theorem \ref{thm:canon RG even-cycles}]
  Fix $2\leq k\in \bbn$ and $r=2k-1$, let $d=d(k)>0$ be the constant
  output by Lemma~\ref{lem:multicolour-paths}, and let $\rho$, $c>0$ be the
  constants output by Theorem~\ref{thm:can-list-RG} with inputs
  $H=C_{2k}$, $r$ and $d$.  Fix $p=\omega\left(n^{-1+1/(2k-1)}\log
    n\right)$ and let $\mu= pn^{1-1/(2k-1)}\left(\log
    n\right)^{-1}=\omega(1)$. Furthermore, let $p_1=\sqrt{\mu}
  \cdot n^{-1+1/(2k-1)}=\omega\left(n^{-1+1/(2k-1)}\right)$ and let
  $p_2=p_2(n)$ be such that $p=p_1+p_2(1-p_1)$.  Note that
  $p_2=\omega(p_1\log n)$.  
  If $\bfg_i\sim \bfg(n,p_i)$ for $i=1,\,2$ with the~$\bfg_i$
  independent, then $\bfg=\bfg_1\cup\bfg_2\sim \bfg(n,p)$.  Let~$A$ be
  the event that there is an edge 
  colouring $\chi:E(\bfg_1)\cup E(\bfg_2)\rightarrow \bbn$ with no
  copies of~$C_{2k}$ which are canonical with respect to $\sigma$ for some
  ordering $\sigma$ of $V(C_{2k})$. We will show that the event~$A$
  a.a.s.\ does not happen, and hence a.a.s.\ $\bfg \canarrow C_{2k}$.
		
  Let $\cG_1$ be the set of $n$-vertex graphs $G_1$ such that
  $e(G_1)\leq 2p_1\binom{n}{2}$ and such that for any colouring
  $\chi_1:E(G_1)\rightarrow \bbn$, there is either a monochromatic
  copy of $C_{2k}$, lexicographic copies of $C_{2k}$ with respect to
  all orderings of the vertex set of $C_{2k}$, or the rainbow focused
  graph $\Gamma(2k,G_1,\chi_1)$ is $(\rho,d)$-dense. By Lemmas
  \ref{lem:RG - properties} \ref{RG: edge count} and
  \ref{lem:multicolour-paths}, we have that a.a.s.\ $\bfg_1\in \cG_1$
  and so
  \[\Pr[ A]\leq \Pr[ A\mid\bfg_1\in \cG_1]+\Pr[\bfg_1\notin \cG_1] \leq
    \max\left\{\Pr[ A\mid\bfg_1=G_1]:G_1\in \cG_1\right\}+o(1).\]
  Hence it suffices to show that for any fixed $G_1\in \cG_1$, we have
  $\Pr[ A\mid\bfg_1=G_1]=o(1)$.
		
  Fix $G_1\in \cG_1$. Now for any \emph{partition} $\cP$ of
  $E(G_1)$, we say an edge colouring $\chi_1:E(G_1)\rightarrow \bbn$
  \emph{induces} $\cP$ if the colour classes given by $\chi_1$ are
  exactly the parts of the partition $\cP$. Furthermore, for a
  partition $\cP$ of $E(G_1)$, let $A(\cP)$ be the event that there is
  some colouring $\chi:E(G_1)\cup E(\bfg_2)\rightarrow \bbn$ with no
  copies of~$C_{2k}$ which are canonical with respect to some ordering
  $\sigma$ and such that $\chi|_{E(G_1)}$ induces $\cP$. We claim that
  it is enough to show that for any partition $\cP$ of $E(G_1)$, we
  have
  \begin{equation} \label{eq:small prob for fixed G1 col}
    \Pr[A(\cP)\mid\bfg_1=G_1] \leq \exp(-cp_2n^2).
  \end{equation}
  Indeed, given that $\bfg_1=G_1$, the event $A$ is contained in the
  event $\bigcup A(\cP)$ where the union goes over all possible
  partitions $\cP$ of $E(G_1)$ and so if \eqref{eq:small prob for
    fixed G1 col} holds, appealing to the union bound, we have that
  \[\Pr[ A\mid\bfg_1=G_1]\leq \sum_{\cP} \Pr[ A(\cP)\mid\bfg_1=G_1]\leq
    \sum_{\cP} \exp(-cp_2n^2)\leq \exp(4p_1n^2\log n-cp_2n^2)=o(1),\]
  where we used that the number of partitions of $E(G_1)$ is at most
  $\exp(e(G_1)\log(e(G_1))\leq \exp(4p_1n^2\log n)$ as $G_1\in \cG_1$
  and the fact that $p_2=\omega(p_1\log n)$. (We remark that it is here where we need the extra log-factor in our probability, all other parts of our argument would work with $p=\omega(n^{-1+1/(2k-1)})$.)
		
  It remains to prove \eqref{eq:small prob for fixed G1
    col}. So fix some partition $\cP$ of $E(G_1)$ and let
  $\chi_1:E(G_1)\rightarrow \bbn$ be some colouring that induces
  $\cP$. It suffices to bound the probability that there is a
  colouring $\chi:E(G_1)\cup E(\bfg_2)\rightarrow \bbn$ with no
  copies of~$C_{2k}$ which are canonical with respect to some ordering~$\sigma$ and such that $\chi|_{E(G_1)}=\chi_1$. Indeed, for any
  colouring $\chi$ of $E(G_1)\cup E(\bfg_2)$ such that
  $\chi|_{E(G_1)}$ induces $\cP$, we can permute the colours so that
  $\chi|_{E(G_1)}=\chi_1$. Now if $G_1$ has a monochromatic copy of
  $C_{2k}$ or lexicographic copies of $C_{2k}$ with respect to all
  orderings, when coloured by $\chi_1$, we have that
  $\Pr[A(\cP)\mid\bfg_1=G_1]=0$ and we are done. So suppose that this is
  not the case and hence, as $G_1\in \cG_1$, we have that
  $\Gamma:=\Gamma(2k,G_1,\chi_1)$ is $(\rho,d)$-dense. Now for each
  edge $e=uv\in \Gamma$, assign a list $\cL(e)$ of the $r=2k-1$
  colours that feature in a rainbow copy of $P_{2k}$ in $G_1$ with
  endpoints $u$ and $v$. With probability at least $1-\exp(-cp_2n^2)$,
  the graph~$\bfg_2$ satisfies the conclusion of Theorem \ref{thm:can-list-RG}
  with respect to the lists $\cL$ and we claim that this implies~\eqref{eq:small prob for fixed G1 col}, as desired. Indeed, for any
  outcome $G_2$ of $\bfg_2$ that satisfies the conclusion of Theorem~\ref{thm:can-list-RG}, and any colouring $\chi:E(G_1)\cup
  E(G_2)\rightarrow \mathbb{N}$ with $\chi|_{E(G_1)}=\chi_1$, we have
  that if $\chi(e)\notin \cL(e)$ for some edge in $E(G_2)\cap \Gamma$,
  then we have a copy of $C_{2k}$ in $G_1\cup G_2$ which is
  rainbow. However, if $\chi(e)\in \cL(e)$ for all edges in
  $E(G_2)\cap \Gamma$ then there is a canonical copy of $C_{2k}$ with
  respect to $\sigma$ for all orderings $\sigma$ in $\Gamma\cap G_2$,
  as $G_2$ is a realisation of $\bfg_2$ satisfying the conclusion of
  Theorem~\ref{thm:can-list-RG}. This concludes the proof of the
  theorem.
\end{proof}
	
\section{Even cycles in random graphs} \label{sec:cycles}
	
Here we give some auxiliary results on even cycles, which will be
useful in our proof of Lemma~\ref{lem:multicolour-paths}. Firstly, in
Section~\ref{sec:intersect}, we give an upper bound on the number of
even cycles that intersect a given vertex set more than twice. We then
explore Tur\'an properties of random graphs with respect to cycles
in Section~\ref{sec:turan} and finally, in Section~\ref{sec:orient},
we discuss finding directed even cycles in orientations of~$\bfg(n,p)$.
	
\subsection{Even cycles intersecting small vertex sets in many
  vertices} \label{sec:intersect}
	
The following lemma will be useful in our proof of Lemma
\ref{lem:multicolour-paths} to upper bound the number of unlabelled
copies of~$C_{2k}$ of a certain unwanted form. 
	
\begin{lemma} \label{lem:intersecting upper bound} For any integer
  $k\geq 2$ and any $0<\zeta\leq \frac{1}{2k^2}$, the following holds
  a.a.s.\ in $\bfg\sim \bfg(n,p)$ when
  $p=\omega\left(n^{-1+1/(2k-1)}\right)$. For any $U\subseteq V(\bfg)$
  such that $|U|\leq \zeta n$, the number of unlabelled copies
  of~$C_{2k}$ in~$\bfg$ that intersect $U$ in more than one vertex is
  at most $k\zeta^2n^{2k}p^{2k}$.
\end{lemma}
	
\begin{proof}
  This proof is similar to~\cite[Lemma 4.1]{ABCDJMRS}. First, we
  introduce some (local) notation. Let~$Z$ be the random variable
  that counts the number of labelled copies of $C_{2k}$ in $\bfg$. For
  each $U\subseteq V(\bfg)$ and $i\in \{0,1,\ldots,2k\}$, let~
  $Z^{U}_{i}$ be the random variable that counts the number of labelled
  copies of $C_{2k}$ that intersects $U$ in exactly $i$ vertices. Set
  $Z^{U}_{\geq i} := \sum_{j\geq i}{Z^{U}_{j}}$.
  Our goal is to show that the event
  \begin{align}\label{eq:intersec-counting} \forall
    \, U \in {n \choose \zeta n}\,,\,\, Z^U_{\geq 2} & \leq 4k^2
                                                       \zeta^2 \, p^{2k} n^{2k}
  \end{align}
  a.a.s.\ occurs. Notice that this result is (essentially)
equivalent to our desired conclusion. Indeed, assume
that~\eqref{eq:intersec-counting} occurs and let $U\subseteq 
V(\bfg)$ with $|U| \leq \zeta n$ be given.  Since increasing the
size of~$U$ only increases the number of cycles intersecting 
it, we can assume the worst case $|U| = \zeta n$.  Now, 
$Z^U_{\geq 2}/{\textrm{aut}}(C_{2k})=Z^U_{\geq 2}/4k$ is the number
of unlabelled copies of~$C_{2k}$ in~$\bfg$ that intersect~$U$
in more than one vertex.  In particular,
if~\eqref{eq:intersec-counting} holds, then the number of such copies
of~$C_{2k}$ is at most $k\zeta^2 \, p^{2k} n^{2k}$, as desired.
It remains to show that~\eqref{eq:intersec-counting} a.a.s.\
occurs. Note that $Z^{U}_{\geq 2} = Z - Z^{U}_{0} - Z^{U}_{1}$.
Let $\epsilon=\zeta^3$.  We shall prove that there is $c =
c(k,\zeta)>0$ such that
  \begin{align}\label{eq:concentr-cycles}
    \Pr[Z >
    (1+\eps) \, p^{2k} n^{2k} ] & = o(1), \\
    \Pr[Z^{U}_0 < (1-\eps) \,
    (1-\zeta)^{2k} \, p^{2k} n^{2k} ] & \leq \exp( - c \, pn^2 )
                                        \label{eq:concentr-0-cycles}\\
    \intertext{and}
    \Pr[Z^{U}_1 < (1-\eps) \, 2k \,
    \zeta (1-\zeta)^{2k-1} \, p^{2k} n^{2k} ] & \leq 2k \, \exp( - c
                                                \, pn^2), \label{eq:concentr-1-cycles}
  \end{align}
  for any $U \in {n \choose \zeta n}$.  Suppose the event
  in~\eqref{eq:concentr-cycles} fails as do the events
  in~\eqref{eq:concentr-0-cycles} and~\eqref{eq:concentr-1-cycles} for
  a given~$U$.  Then, as we shall check shortly, we have that
  $Z^U_{\geq 2} \leq 4k^2\zeta^2 \, p^{2k} n^{2k}$ for that set~$U$,
  which is the desired inequality in~\eqref{eq:intersec-counting}.
  Thus, the fact that~\eqref{eq:intersec-counting} holds a.a.s.\
  follows from~\eqref{eq:concentr-cycles},
  \eqref{eq:concentr-0-cycles} and~\eqref{eq:concentr-1-cycles} by the
  union bound, because
  $o(1) + (2k+1){ n \choose \zeta n} \exp(-c \, pn^2) = o(1)$, where
  we have used that $pn^2 = \omega(n)$.  It remains to check that if
  the events in~\eqref{eq:concentr-cycles},
  \eqref{eq:concentr-0-cycles} and~\eqref{eq:concentr-1-cycles} all
  fail, then $Z^U_{\geq 2} \leq 4k^2\zeta^2 \, p^{2k} n^{2k}$.
  Suppose those events do fail.  Then
  \begin{align*} Z^{U}_{\geq 2} = Z - Z^{U}_{0} -
    Z^{U}_{1} \leq \left((1+\eps) - (1-\eps) \, (1-\zeta)^{2k} -
    (1-\eps) \, 2k \, \zeta (1-\zeta)^{2k-1}\right)p^{2k} n^{2k}.
  \end{align*}
  Applying Boole's inequality twice and the fact that
  $\eps = \zeta^3$ and $\zeta \leq1/{2k^2}$, we see that 
  \begin{align*}
    (1+\eps)  &- (1-\eps) \, (1-\zeta)^{2k} -  (1-\eps) \, 2k \, \zeta (1-\zeta)^{2k-1}  \\
              &\quad =  1 - (1-\zeta)^{2k} - 2k \zeta(1-\zeta)^{2k-1} + \eps ( 1 + (1-\zeta)^{2k} + 2k \zeta(1-\zeta)^{2k-1} ) \\
              &\quad \leq 1 - (  1 - 2k\zeta) -2k \zeta (1 - (2k-1) \zeta) + \eps(2+2k\zeta) \\
              &\quad \leq 2k\zeta^2(2k-1)+3\epsilon\\
              &\quad \leq 4k^2 \zeta^2, 
  \end{align*}
  and hence $Z^U_{\geq 2} \leq 4k^2\zeta^2 \, p^{2k} n^{2k}$, as
  required. 
  
  Thus, it is sufficient to prove that~\eqref{eq:concentr-cycles}
  holds and that~\eqref{eq:concentr-0-cycles} and
  \eqref{eq:concentr-1-cycles} hold for any
  $U \in {n \choose \zeta n}$.  These results follow by standard
  concentration arguments.  More exactly, one can show
  \eqref{eq:concentr-cycles} by the second moment method, while
  crucially one can show \eqref{eq:concentr-0-cycles} and
  \eqref{eq:concentr-1-cycles} using Janson inequalities (which give
  us a good exponential decay). Upper and lower tails for subgraph
  counts have been extensively studied (see for
  instance~\cite{JLR00}), and noticing that $Z_0^U$ is the random
  variable counting labelled copies of $C_{2k}$ in (a copy of)
  $\bfg((1-\zeta)n,p)$, we just need to
  justify~\eqref{eq:concentr-1-cycles}.  We introduce some formalism
  for the sake of precision.  Let~$C_{2k}^0$ be a $2k$-cycle with
  $V(C_{2k}^0)=[2k]$.  Also, let~$K_n^0$ be the complete graph with
  $V(K_n^0)=[n]$, and suppose our random graphs~$\bfg\sim\bfg(n,p)$
  are subgraphs of~$K_n^0$.  A labelled copy of~$C_{2k}$ in~$K_n^0$ is
  an injective map $\phi:[2k]\to[n]$.  Thus, $Z_1^U$~is the number of
  such~$\phi$ such that~$\phi(C_{2k}^0)$ is a subgraph of~$\bfg$ and
  $|V(\phi(C_{2k}^0))\cap U|=1$.  For convenience, we let
  $Z_1^U = \sum_{s \in [2k]}{Z_1^{U;s}}$, where~$Z_1^{U;s}$ is the
  number of~$\phi$ with $\phi(C_{2k}^0)\subset\bfg$
  such that
  $V(\phi[C_{2k}^0])\cap U = \{\phi(s)\}$. Note that, by symmetry, the
  random variables $Z_1^{U;s}$ ($s\in[2k]$) have the same
  distribution. Thus, it is enough to show that
\begin{linenomath} \begin{align}\label{eq:concentr-1_1-cycle} \Pr[
    Z_1^{U;1} < (1-\eps) \, \zeta (1-\zeta)^{2k-1} \, p^{2k} n^{2k} ]
    \leq \exp(- c \, pn^2) .
\end{align} \end{linenomath}
Following the notation of~\cite[Chapter 2]{JLR00}, we have that
\begin{linenomath} \begin{align} \label{eq:Janson_1-1-cycles} \Pr[
    Z_1^{U;1} < (1-\eps) \, \mu ] \leq \exp\left(- \frac{\eps^2
        \mu^2}{2 \Delta + \mu} \right),
\end{align} \end{linenomath}
where $\mu = \Ex[ Z_1^{U;1} ] = (1+o(1)) \zeta (1-\zeta)^{2k-1} p^{2k} n^{2k}$, and 
\begin{linenomath}
  \begin{align*}
    \Delta := \Big(\sum_{\psi \sim \phi }{ \Pr[\psi[C_{2k}^0]
    \cup \phi[C_{2k}^0] \subseteq \bfg ] }\Big)/2
  \end{align*}
\end{linenomath}
where $\psi \sim \phi$ indicates that the sum ranges over all
pairs~$(\psi,\phi)$ of distinct labelled copies of~$C_{2k}^0$
in~$K_n^0$ which
send vertex~$1$ to~$U$ and
$|E(\psi[C_{2k}^0])\cap E(\phi[C_{2k}^0])|\geq1$.  Note that
\begin{equation}\label{eq:delta_1-1-cycle}
  \Delta = O(p^{4k-1} n^{4k-2}),
\end{equation}
since  the expectation of the number of pairs of labelled cycles of length $2k$ that intersect in at least one edge in $G(n,p)$ without the restriction of them containing exactly  one vertex of $U$ is $O(p^{4k-1} n^{4k-2})$. 
 Furthermore,
because of our hypothesis on~$p$, we have that $\mu=o(p^{4k-1}n^{4k-2})$, and 
hence the denominator in~\eqref{eq:Janson_1-1-cycles} is
at most $O(p^{4k-1}n^{4k-2})$.  Taking this into account and plugging in the value
of~$\mu$ 
into~\eqref{eq:Janson_1-1-cycles}, we
conclude~\eqref{eq:concentr-1_1-cycle}, as desired.
\end{proof}	
	
\subsection{Tur\'an properties of random graphs} \label{sec:turan} We
will use the following theorem established by Haxell, {\L}uczak and
the second author of the current paper~\cite{HKL95}.
	
\begin{theorem}[Relative Turán Theorem: even-cycle
  case] \label{thm:turan} Let $2\leq k\in \bbn$ and $\delta>0$ be
  arbitrary.  If $p=p(n)$ is such that
  $p=\omega\left(n^{-1+1/(2k-1)}\right)$, then a.a.s.\ $\bfg\sim
  \bfg(n,p)$ has the property that any subgraph $H\subseteq \bfg$ with
  $e(H)\geq \delta np^2$ contains a copy of $C_{2k}$.
\end{theorem}
	
We remark that more precise results \cite{KKS98,MS16} for even cycles
have since been obtained and that Theorem~\ref{thm:turan} falls into
the general program of establishing the extremal densities of $H$-free
subgraphs of $\bfg(n,p)$. This line of research culminated in
breakthrough results \cite{CG16,S16} providing the full picture for
all graphs $H$ and subsequent elegant alternative proofs
\cite{BMS14,ST15} using the method of hypergraph containers (see the
next section). We refer the reader to the survey of R\"odl and
Schacht~\cite{RS13} and Conlon~\cite{conlon14:_combin}
for more on the topic.

\subsection{Directed cycles in orientations of random graphs}
\label{sec:orient}
We will be interested in finding copies of directed cycles in
arbitrary orientations of $\bfg(n,p)$. This problem has been studied
by Cavalar \cite{C} (see also~\cite{BCKN, barros24:_direc_ramsey}).
In particular, it is known that when~$\vec{H}$ is an acyclic
orientation of a complete graph or a cycle, 
then, except for the transitive triangle, 
$n^{-1/m_2(H)}$ is the threshold for the property that any
orientation of $\bfg(n,p)$ contains a copy of~$\vec{H}$.
Here we require the
following lemma, which is a slight strengthening of a theorem of
Cavalar~\cite{C} (see also~\cite{barros24:_direc_ramsey}).  For
simplicity, we restrict ourselves to cycles, as this is the only case
we need here. 
	
It will be convenient to introduce the following definition.  Let
$3\leq\ell\in \bbn$ and~$x\geq0$ be given. We say that a graph~$G$ has
the \textit{$(C_\ell,x)$-orientation property} if for any acyclic
orientation~$\vec C_\ell$ of~$C_\ell$ and any orientation~$\vec G$
of~$G$, the number of labelled copies of~$\vec C_\ell$ in~$\vec G$ is
at least~$x$.
	
\begin{lemma}
  \label{lem:orient}
  Let $3\leq\ell\in \bbn$ and $\epsilon>0$ be given and suppose
  $p=p(n)$ is such that $p=\omega\big(n^{-1+1/(\ell-1)}\big)$.  Then
  there is~$\eta>0$ such that for every~$\nu>0$ a.a.s.\ $\bfg\sim
  \bfg(n,p)$ has the property that for every $W\subseteq V(\bfg)$ with
  $|W|\geq \nu n$, every subgraph~$G\subseteq \bfg[W]$ with
  \begin{equation}
    \label{eq:1}
    e(G)\geq\left(\pi(K_{2^{\ell-1}})+\eps\right)
    p\binom{|W|}2      
  \end{equation}
  has the $\big(C_\ell,\eta(\nu pn\big)^\ell)$-orientation property.
\end{lemma}
	
Lemma~\ref{lem:orient} can be proved with the sparse regularity
method.  A sketch of the proof is given in
Appendix~\ref{sec:pfoforient}.
	
\section{Locally dense rainbow focused graphs} \label{sec:rainbow
  paths}
	
In this section, we prove Lemma \ref{lem:multicolour-paths}. Before
embarking on this, let us sketch the main ideas of the proof. Suppose
we have $\bfg\sim \bfg(n,p)$ as in Lemma \ref{lem:multicolour-paths}
and that $\chi:E(\bfg)\rightarrow \bbn$ is some edge colouring of
$\bfg$. We can further assume that $\chi$ has no monochromatic copy
of~$C_{2k}$ as otherwise the proof is done and so we aim to prove that
either there are lexicographic copies of~$C_{2k}$ for all orderings of
its vertices, or the third conclusion of the lemma holds: that is, the
rainbow focused graph is locally dense.  Firstly, through an
application of Lemma~\ref{lem:well-distributed}, we will show that in
order to prove that the rainbow focused graph is locally dense, it
suffices to find $d|S|^2n^{2k-2}p^{2k-1}$ rainbow copies of $P_{2k}$
with endpoints in $S$ for every $S \subseteq V(\bfg)$ with $|S| = \rho
n$ (see Claim~\ref{clm:many paths}). To simplify this discussion
(whilst still conveying the key ideas of the proof), let us focus on
the case $\rho=1$ and so $S=V(\bfg)$.  We begin by classifying the
vertices of $\bfg$, depending on the profile of colours in their
neighbourhood. In particular, we will be interested in whether there
is some constant proportion of the edges incident to a given vertex
that have all been assigned the same colour by~$\chi$. This leads us to
the following definition.
	
\begin{definition} \label{def:heavy} Let $G$ be an $n$-vertex graph
  and $\chi:E(G)\rightarrow \bbn$ an edge colouring of $G$. Given
  $\alpha>0$ and $c\in \bbn$, we say that
  \begin{itemize}
  \item $v\in V(G)$ is \emph{$(\alpha\,; c)$-heavy} if there
    are at least $\alpha d_G(v)$ edges~$uv$ with $\chi(uv)
    = c$;
  \item $U\subseteq V(G)$ is \emph{$\alpha$-heavy} if every $v\in
    U$ is $(\alpha\,;c)$-heavy for some $c\in \bbn$;
  \item $\lambda:U\rightarrow \bbn$ is \emph{$\alpha$-heavy} if
    for every $c\in \bbn$ all vertices in $\lambda^{-1}(c)\subseteq
    U$ are $(\alpha \,;c)$-heavy.
  \end{itemize}
  Furthermore, a vertex $v\in V(\bfg)$ is \emph{$\alpha$-rainbow} if
  $v$ is not $(\alpha\,;c)$-heavy for \emph{any} $c\in \bbn$.
\end{definition}

Now we discuss how to find the desired rainbow copies of $P_{2k}$, but
to give the precise intuition we start by analysing two extreme cases,
depending whether all the vertices are heavy or all the vertices are
rainbow.

Suppose first that for some small $\alpha<1/2k$, all
vertices $v\in V(\bfg)$ are $\alpha$-rainbow.  In this case, we can
greedily form rainbow paths, up to length $2k-1$, which is our
goal. Indeed, suppose we have some rainbow path $P$ with $\ell < 2k$
vertices such that $P$ features a collection $C\subseteq \bbn$ of
$\ell-1$ colours and $v\in V(\bfg)$ is an endpoint of $P$. As $v$ is
$\alpha$-rainbow, the number of edges incident to $v$ which are
coloured with colours belonging to $C$ is less than $(\ell-1)\alpha
d_\bfg(v)\leq (2k-1)\alpha d_{\bfg}(v)$. Thus as
$\alpha<1/2k$, we have at least $(1/2k)d_{\bfg}(v)$
choices of edge that extend $P$ to a rainbow path with $\ell+1$
vertices. A simple induction then yields the desired number of rainbow
paths with~$2k$ vertices, using that all vertices~$v$ have degree
$d_\bfg(v)=(1\pm o(1))pn$ by Lemma~\ref{lem:RG -
  properties}~\ref{RG: degrees}.
	
At the other extreme, let us now suppose that $V(\bfg)$ is
$\alpha$-heavy and that $\lambda:V(\bfg)\rightarrow \bbn$ is some
$\alpha$-heavy colour allocation. Our approach in this case is again
to greedily extend rainbow paths, always using a heavy colour at the
endpoint of the path to extend it. That is, if $P$ is a rainbow path
with $\ell < 2k$ vertices and $v$ is an endpoint of $P$, we will
extend $P$ by using one of the $\alpha d_\bfg(v)$ edges incident to
$v$ with colour $\lambda(v)$. Of course, in order for this path to be
rainbow, we need that $P$ does not feature the colour $\lambda(v)$
already. To impose this, we need to add extra conditions on our
rainbow paths with the foresight that we do not want to end up in such
a situation. In more detail, suppose we have a rainbow path $P$ with
$\ell$ vertices coloured with a set $C\subset \bbn$ of $\ell-1$
colours and an endpoint $v$ with the added condition that
$\lambda(v)\notin C$. Then we extend along edges $e=vw$ coloured with
$\lambda(v)$ but we also impose that $\lambda(w)\notin C\cup
\{\lambda(v)\}$. This then allows us to continue the inductive process
and extend the rainbow path from $w$. Now in order to get the correct
number of rainbow paths, we need to show that this extra condition of
dodging endpoints~$w$ such that $\lambda(w)\in C\cup \{\lambda(v)\}$
does not restrict the number of choices significantly at each step of
the greedy process. Furthermore, we also need to guarantee that there
are still at least, say, $(\alpha/2)d_\bfg(v)$ edges $e=vw$
such that $\chi(e)=\lambda(v)$ and $\lambda(w)\notin C\cup
\{\lambda(v)\}$. This can be achieved (for almost all $v$) and follows
from proving that, say, $|\lambda^{-1}(c)\cap
N_\bfg(v)|<(\alpha/4k)pn$ for all $c\in \bbn$, which is the
content of the preparatory Lemma~\ref{lem:heavy}~\ref{item:heavy1}
below (the proof of which relies on the assumption that there is no monochromatic $C_{2k}$).

It remains to address the case where we have a mix of heavy and
rainbow vertices. That is, there is some maximal $\alpha$-heavy subset
$U\subset V(\bfg)$ and some $\alpha$-heavy colour allocation
$\lambda:U\rightarrow \bbn$, and also there are $\alpha$-rainbow
vertices outside~$U$.

We will extend the $\alpha$-heavy colour allocation $\lambda$ to an
$\alpha$-rainbow vertex $v\in V(\bfg)$ by setting $\lambda(v)=\star$
for some $\star\notin \bbn$. Now consider the greedy approach as
before and suppose we have a rainbow copy $P$ of $P_\ell$ with colours
$C\subseteq \bbn$ and endpoint $v$. If $\lambda(v)=c\in \bbn$, then we
extend the path $P$ with the edges coloured with colour $c$ and if $\lambda(v)=\star$, we extend with
edges that avoid the colours in $C$. Moreover, similarly to the
previous case where all vertices were heavy, we impose that the new
edge $e=vw$ is such that $\lambda(w)\notin C$. This almost works to
produce rainbow paths iteratively as before but note that in order to
guarantee that a path given by this process is indeed rainbow, we also
need that $\lambda(w)\neq \chi(e)$ as otherwise our path may contain
two consecutive edges of the same colour. In the case that $v$ is
$(\alpha\,;c)$-heavy for some $c\in \bbn$, we can impose that
$\lambda(w)\neq \chi(e)=\lambda(v)$ as in the previous
paragraph. However, in the case that $\lambda(v)=\star$, we have a
problem. Indeed, it could be the case that all edges $e=vw$ incident
to~$v$ such that $\chi(e)\notin C$ are such that
$\chi(e)=\lambda(w)$. What if, for instance, every vertex $u\in U$ has
all incident edges coloured $\lambda(u)$?  In fact, 
Lemma~\ref{lem:heavy}~\ref{item:heavy3} below shows that this cannot
be the case (again under the assumption of no monochromatic $C_{2k}$) and that there must be some subset $W\subseteq U$ such
that every vertex $w\in W$ has some significant proportion of its
incident edges which are coloured in colours that are \emph{not}
$\lambda(w)$. This suggests that we carry out our iterative procedure
for building paths imposing that we only use these edges not coloured
$\lambda(w)$ when our paths enter some $w\in W$. That is, we will only
consider rainbow paths $P$ with endpoints in $W$ if the path with
endpoint $w\in W$ has its last edge coloured in some colour other than
$ \lambda(w)$. This will ensure that we can extend such a path with
edges coloured $\lambda(w)$ as before.
	
In order for this to work, we will need to run our iterative process
with respect to disjoint \emph{layers} of vertices
$U_1,\dots,U_{k-1}$ and setting $U_0$ to be $S$ (recall that  $S=V(\bfg)$ for the purpose of this sketch). These layers are designed in such a way (Claim
\ref{clm:Ui} in the proof of Lemma \ref{lem:multicolour-paths}) that
rainbow paths with endpoints in $U_\ell$ can be extended to rainbow
paths with endpoints in $U_{\ell-1}$ with the added foresight
discussed above that will allow the rainbow paths to be extended
further to $U_{\ell-2}$ and so on. Consequently, we will only count
rainbow paths which respect this vertex partition but this will only
affect the value of our density constant~$d=d(k)$.
	
Finally whilst we can define each $U_{\ell}$ in such a way that each
vertex $v$ in $U_{\ell}$ is adjacent to many edges $vw$ with vertices
$w$ in $U_{\ell-1}$ that are \emph{good} in the sense that
$\chi(vw)\neq \lambda(w)$, we have one final hurdle which prevents us
from finding rainbow paths. We need to find edges in $U_{k-1}$ that
can start our paths and be extended \emph{by both their endpoints} to
$U_{k-2}$ in a rainbow fashion. It turns out that this is not always
possible. Indeed consider the following instructive example. Suppose
every vertex~$v$ in $U_{k-1}$ is $\alpha$-heavy and in fact all
edges from~$v$ to~$U_{k-2}$ 
are coloured~$\lambda(v)$.  Moreover,
suppose that every edge $e=vv'$ contained in $U_{k-1}$ shares its
colour with one of the heavy colours of its endpoints. That is,
$\chi(e)$ is either $\lambda(v)$ or $\lambda(v')$. In such a
colouring, one has that there are \emph{no} rainbow paths of length
three with endpoints in $U_{k-2}$ and an edge in $U_{k-1}$. Our proof
will show that such a colouring is essentially the only type of
colouring that prevents us finding the right number of rainbow
paths. In the case that such a colouring occurs, we will use Lemma
\ref{lem:orient} to prove the existence of lexicographic copies of
$C_{2k}$ with respect to all orderings. This is the content of
Lemma~\ref{lem:lex copies} below. 

First though, given an $\alpha$-heavy set $U\subseteq V(\bfg)$ and an
$\alpha$-heavy allocation $\lambda:U\rightarrow \bbn$, we need the
following definition for an edge $e=uv\in \bfg[U]$:
\begin{linenomath}
  \begin{align*}
  &e=uv\in E(\bfg[U])\text{ is \emph{balanced} if }\lambda(u)\neq
    \lambda(v),\text{ and}\\
  &e=uv\in E(\bfg[U])\text{ is \emph{$u$-light} if }\chi(e)\neq\lambda(u).
  \end{align*}
\end{linenomath}
The concepts of balanced and light edges appear in the following
result, which will be used in the proof of
Lemma~\ref{lem:multicolour-paths} to define the vertex sets $U_i$ in
an appropriate fashion as discussed.
	
\begin{lemma} \label{lem:heavy} For all $\alpha>0$ and
  $2\leq k\in \bbn$ the following holds a.a.s.\ in $\bfg\sim\bfg(n,p)$
  when $p=\omega\left(n^{-1+{1}/(2k-1)}\right)$. For any monochromatic
  $C_{2k}$-free edge colouring $\chi$ of $\bfg$, any $\alpha$-heavy
  $U\subseteq V(\bfg)$ with $|U|\geq \alpha n$ and any $\alpha$-heavy
  colour allocation $\lambda:U\rightarrow \bbn$, we have the
  following.
  \begin{enumerate}[{\rm(1)}]
  \item \label{item:heavy1}For all but at most $\alpha n$ vertices $v$
    of $\bfg$, we have $|N_\bfg(v)\cap \lambda^{-1}(c)|\leq \alpha
    p |U|$ for any $c\in \bbn$.
  \item \label{item:heavy2} $\bfg[U]$ contains at least
    $(1-2\alpha)p|U|^2/2$ balanced edges.
  \item \label{item:heavy3} There exists $W\subseteq U$ such that
    $|W|= \tfrac{1}{8}|U|$ and each vertex $w\in W$ is incident to at
    least $\tfrac{p}{5}|U|$ $w$-light edges in $\bfg$ whose other
    endpoint is in $U\setminus W$.
  \end{enumerate}
\end{lemma}

\begin{proof}
  Note that we can assume that $0<\alpha<\tfrac{1}{100}$ as the
  statement is stronger for smaller $\alpha$. Fix $k\in \bbn$,
  $\delta:=\tfrac{\alpha^3}{32}$ and
  $\eps:=\tfrac{\alpha^2}{4(1+\alpha)}$ and let $G$ be any outcome of
  $\bfg(n,p)$ that satisfies the conclusions of Lemma \ref{lem:RG -
    properties} and Theorem \ref{thm:turan}. As these properties all
  hold a.a.s.\ in $\bfg(n,p)$, it suffices to show that any such $G$
  satisfies the desired conclusions. So fix such a $G$ and
  furthermore, fix a monochromatic $C_{2k}$-free edge colouring $\chi$
  of $G$, some $\alpha$-heavy set $U$ with $|U|\geq \alpha n$ and some
  $\alpha$-heavy colour allocation $\lambda:U\rightarrow
  \bbn$. Appealing to the conclusion of Lemma~\ref{lem:RG -
    properties}~\ref{RG: degrees}, we have that each $u\in U$ has
  $d_G(u)\geq (1-\eps)pn$ and so has at least $\alpha d_G(u)\geq
  \tfrac{\alpha}{2} p n$ incident edges coloured $\lambda(u)$ by
  $\chi$. Therefore, for each $c\in \bbn$, we have
  $|\lambda^{-1}(c)|\leq \frac{\alpha}{8}|U|$.  Indeed if this
  was not the case then there would be at least $\delta p n^2$ edges
  of $G$ coloured $c$ by $\chi$ and by (the assumed conclusion of)
  Theorem \ref{thm:turan}, there is a copy of $C_{2k}$ in $G$ coloured
  $c$, contradicting the fact that $\chi$ is monochromatic
  $C_{2k}$-free.
  
  In order to prove part~\ref{item:heavy1} of the lemma, we
  partition $U$ as $U=S_1\cup\dots\cup S_t$ for some $t\leq
  \tfrac{4}{\alpha}$ such that each~$S_i$ has size $|S_i|\leq
  \tfrac{\alpha}{2}|U|$ and for each $c\in \bbn$, there is some $i\in
  [t]$ such that $\lambda^{-1}(c)\subseteq S_i$. In other words, we
  group the sets $\lambda^{-1}(c)$ for $c\in \bbn$ together in order
  to get at most $\tfrac{4}{\alpha}$ groups, each of which has size at
  most $\tfrac{\alpha}{2}|U|$. This can be done greedily by adding
  sets $\lambda^{-1}(c)$ to $S_1$ until $S_1$ has size at least
  $\tfrac{\alpha}{4} |U|$, and then moving on to fill~$S_2$ and
  repeating until all the sets $\lambda^{-1}(c)$ have been placed.
  Now by (the assumed conclusion of) Lemma~\ref{lem:RG -
    properties}~\ref{RG:bad vts}, we have that for each $1\leq i \leq
  t$, there is 
  a set $B_i\subseteq V(G)$ with $|B_i|\leq \eps^2 n$ such that every
  vertex $v\in V(G)\setminus B_i$ has $d_G(v,S_i)\leq \alpha
  p|U|$. Note here that even if~$S_i$ has size less than $\eps n$, the
  conclusion of Lemma~\ref{lem:RG - properties}~\ref{RG:bad vts} can
  be applied to a superset of~$S_i$. Letting $B=\bigcup_{i=1}^tB_i$, for
  any $c\in \bbn$ and any vertex $v\in V(G)\setminus B$ we have that
  $|N_G(v)\cap \lambda^{-1}(c)|\leq \alpha p |U|$ as
  $\lambda^{-1}(c)\subseteq S_i$ for some $i\in t$. Noting that
  $|B|\leq t\eps^2n\leq \alpha n$ then completes the proof of~\ref{item:heavy1}.

  For part~\ref{item:heavy2} of the lemma, we note that any edge in
  $G[U]$ that is not contained in~$G[S_i]$ for some $i\in [t]$, is
  balanced. Indeed this follows from the fact that each
  $\lambda^{-1}(c)$ is contained in some~$S_i$. Hence the number of
  balanced edges in $G[U]$ is at least
  \[e_G(U)-\sum_{i=1}^te_G(S_i)\geq
    \frac{(1-\eps)p|U|^2}{2}-\frac{t(1+\eps) p
      \left(\frac{\alpha}{2}|U|\right)^2}{2}\geq \frac{(1-2\alpha) p
      |U|^2}{2},\]
  as required, using the assumed conclusion of Lemma~\ref{lem:RG -
    properties}~\ref{RG: edge count} to lower bound $e_G(U)$ and
  upper bound the~$e_G(S_i)$ as well as the upper bounds on~$|S_i|$
  for each $i\in [t]$ and the upper bound on $t$.
		
  For part~\ref{item:heavy3}, we will say that an edge $e=uv$ is
  \emph{light} if it is $u$-light or $v$-light. Note that every
  balanced edge is necessarily light and define $G^*$ to be the
  subgraph of $G[U]$ given by all the light edges. Moreover, for each
  edge $e=uv\in E(G^*)$, we orient the edge so that if it is directed
  from $u$ to $v$ then the edge $e$ is $u$-light.  Note that for some
  edges $e=uv\in G^*$, both orientations are allowed as $\chi(e)\notin
  \{\lambda(v),\lambda(u)\}$; for such edges choose an orientation
  arbitrarily. Now let $V_0\subseteq U$ be the set of vertices $v\in
  U$, such that $d^+_{G^*}(v)\geq \tfrac{2p}{5}|U|$, where
  $d^+_{G^*}(v)$ denotes the number of out-neighbours of $u$ in our
  oriented $G^*$. We claim that $|V_0|\geq\tfrac{1}{6}|U|$. Indeed, if
  this is not the case, then consider a set $V_1\subseteq U\setminus
  V_0$ with $|V_1|=\tfrac{5}{6}|U|$ and $V_2=U\setminus V_1$ so
  that $|V_2|=\tfrac{1}{6}|U|$ and $V_0\subseteq V_2$. We upper bound
  the number of edges in $G^*$ by
  \begin{linenomath}
    \begin{align*}
      e_{G^*}(V_1)+e_{G^*}(V_1,V_2)+e_{G^*}(V_2)
      &\leq \sum_{v\in
        V_1}d^+_{G^*}(v)+e_{G}(V_1,V_2)+e_G(V_2) \\
      &\leq |V_1|\cdot
        \frac{2p}{5}|U|+(1+\eps)p|V_1||V_2|+(1+\eps)\frac{p
        |V_2|^2}{2} \\
      &\leq
        \left(\frac{1}{3}+(1+\eps)\left(\frac{5}{36}+\frac{1}{72}\right)
        \right)p|U|^2<
        \frac{(1-2\alpha)p|U|^2}{2},
    \end{align*}
  \end{linenomath}
  contradicting part~\ref{item:heavy2} of the lemma, as every
  balanced edge is light. We used here the assumed conclusion of
  Lemma~\ref{lem:RG - properties}~\ref{RG: edge count}
  and~\ref{RG:edges between subsets} in $G$.
		
  We then know that $|V_0|\geq \tfrac{1}{6}|U|$ and we fix
  $W'\subseteq V_0$ such that $|W'|=\tfrac{1}{8}|U|+\eps^2n$. By the
  assumed conclusion of Lemma~\ref{lem:RG - properties}~\ref{RG:bad
    vts}, we have that there is some $W\subseteq W'$ with
  $|W|=\tfrac{1}{8}|U|$ and every vertex $w\in W$ has $d_G(w,W)\leq
  d_G(w,W')\leq (1+\eps)p|W'|\leq \tfrac{1}{5}p|U|$. Therefore as
  $W\subseteq V_0$, each vertex $v\in W$ is incident to at least
  $\left(\tfrac{2}{5}-\tfrac{1}{5}\right)p|U|=\tfrac{1}{5}p|U|$
  $v$-light edges whose other endpoint is in $U\setminus W$, as
  required.
\end{proof}
	
We will also need the following lemma which will be used to find
lexicographic copies of $C_{2k}$ in the case that certain colourings
are present. Recall that an edge $e=uv\in G[U]$ {$u$-light} if
$\chi(e)\neq\lambda(u)$ and balanced if $\lambda(u)\neq
\lambda(v)$. Moreover, call an edge $e=uv\in G[U]$ \emph{totally
  light} if $\chi(e)\notin \{\lambda(u),\lambda(v)\}$, that is, if it
is light with respect to \emph{both} of its vertices.
	
\begin{lemma} \label{lem:lex copies} For all positive~$\nu$
  and~$\alpha$ and $2\leq k\in \bbn$ the following holds a.a.s.\ in
  $\bfg\sim\bfg(n,p)$ when $p=\omega\left(n^{-1+{1}/(2k-1)}\right)$.
  For any monochromatic $C_{2k}$-free edge colouring $\chi$ of $\bfg$,
  any $\alpha$-heavy $U\subseteq V(\bfg)$ with $|U|\geq \nu n$ and any
  $\alpha$-heavy colour allocation $\lambda:U\rightarrow \bbn$, we
  have the following.

  Either $\bfg[U]$ contains $2^{-(8k+1)}p|U|^2$ balanced and
  totally light edges, or for all orderings~$\sigma$ of~$V(C_{2k})$,
  the graph $\bfg[U]$ contains a copy of $C_{2k}$ which is
  lexicographic with respect to $\sigma$.
\end{lemma}
\begin{proof}
  Let~$\nu>0$ be fixed.
  Note that we can assume that $0<\alpha<\min\{2^{-16k},\nu\}$ as the
  statement of the lemma is stronger for smaller $\alpha$. Moreover,
  let $\eta>0$ be as output by Lemma~\ref{lem:orient} with input
  $\ell=2k$ and $\eps=2^{-2\ell}$.  Fix $0<\zeta<\eta \nu^{2k}/(8k)$ and
  $\delta=\zeta\alpha/8$ and let $G$ be an instance of $\bfg(n,p)$
  satisfying the conclusions of Theorem \ref{thm:turan} and Lemmas
  \ref{lem:intersecting upper bound}, \ref{lem:orient} and
  \ref{lem:heavy}. It suffices to show that such a $G$ satisfies the
  desired statement as all the listed properties of $G$ occur a.a.s.\
  in~$\bfg(n,p)$. So fix such a $G$ and a monochromatic $C_{2k}$-free
  colouring $\chi$. Further, fix some $\alpha$-heavy vertex subset
  $U\subseteq V(G)$ with $|U|\geq \nu n$ and an $\alpha$-heavy colour
  allocation $\lambda:U\rightarrow \bbn$.
		
  As in the proof of Lemma \ref{lem:heavy}, we can assume that
  $|\lambda^{-1}(c)|\leq \zeta n/2$ for all $c\in \bbn$, as otherwise
  such a colour class would have at least $\delta pn^2$ edges and thus
  contain a copy of $C_{2k}$ by the assumed conclusion of Theorem
  \ref{thm:turan}. As we have assumed that our colouring $\chi$ of $G$
  is monochromatic $C_{2k}$-free, this gives a
  contradiction. Therefore each $\lambda^{-1}(c)$ has size less than
  $\zeta n/2$. Again, as in the proof of Lemma \ref{lem:heavy}, we can
  therefore partition $U$ as $U=S_1\cup\cdots \cup S_t$ for some
  $t\leq {2/\zeta}$ such that each $S_i$ has size $ |S_i|\leq \zeta n$
  and for each $c\in \bbn$, there is some $i\in [t]$ such that
  $\lambda^{-1}(c)\subseteq S_i$.
		
  Now if there are $2^{-(8k+1)}p|U|^2$ balanced edges in $G[U]$
  that are also totally light, we are done. So suppose this is not the
  case. Then by appealing to the assumed conclusion of
  Lemma~\ref{lem:heavy}~\ref{item:heavy2}, there is some subgraph
  $G^*$ of~$G[U]$ such that every edge in $G^*$ is balanced (and therefore
  light) but \emph{not} totally light and
  \[e(G^*)\geq \left(1-2\alpha-2^{-8k}\right)p|U|^2/2\geq
    (1-2^{-4k})p|U|^2/2.\]
  We orient $G^*$ as follows. Each edge $e=uv\in E(G^*)$ is light with
  respect to exactly one of its vertices, say $v$, and we orient $e$
  from $u$ to $v$. Thus each directed edge $\vec{uv}$ of $G^*$ gives
  an edge $e=uv$ in $G$ such that $\chi(e)=\lambda(u)$. Indeed $e$ is
  light with respect to $v$ and so $\chi(e)\neq \lambda(v)$ but $e$ is
  \emph{not} totally light and so we must have that
  $\chi(e)=\lambda(u)$. By the assumed conclusion of Lemma
  \ref{lem:orient} and our definition of $\eps$, we have that our
  orientation of $G^*$ contains $\eta (\nu np)^{2k}$ copies of
  $\vec{H}$ for any acyclic orientation $\vec{H}$ of $C_{2k}$.
		
  Now consider some arbitrary ordering $\sigma$ of $V(C_{2k})$ and let
  $\vec{H}$ be the orientation of $C_{2k}$ which has each edge
  oriented from its first to its second vertex, according to the
  ordering $\sigma$. Note that such an orientation $\vec{H}$ is
  necessarily acyclic. Moreover, any copy of $\vec{H}$ in $G^*$
  corresponds to a copy of $C_{2k}$ in $G$ where each edge $e=uv$ is
  coloured $\lambda(\min\{u,v\})$ where the minimum is taken with
  respect to the ordering $\sigma$. Now if $\lambda$ is injective on
  the vertices of such a copy of $C_{2k}$, then this copy would be
  lexicographic with respect to $\sigma$. In order to guarantee the
  existence of such a lexicographic copy we upper bound the number of
  copies of $C_{2k}$ that have more than one vertex in some set
  $\lambda^{-1}(c)$. In fact, we bound the number of copies of
  $C_{2k}$ with more than one vertex in $S_i$ for some $i\in [t]$,
  which suffices as for each $c\in \bbn$, there is an $i\in [t]$ with
  $\lambda^{-1}(c)\subseteq S_i$. For a fixed $i\in [t]$, by the
  assumed conclusion of Lemma \ref{lem:intersecting upper bound},
  there are at most $k\zeta^2(np)^{2k}$ copies of $C_{2k}$
  intersecting $S_i$ in more than one vertex. Thus, there are at most  
  \[t\cdot k \zeta^2(np)^{2k}\leq2k \zeta(np)^{2k} <\eta (\nu n
    p)^{2k}/4\]
  copies of $C_{2k}$ with a pair of distinct vertices $u$ and $v$ such
  that $\lambda(u)=\lambda(v)$. Hence at least one of the copies of
  our oriented cycle $\vec{H}$ that came from the application of Lemma
  \ref{lem:orient} corresponds to a copy of $C_{2k}$ which is
  lexicographic with respect to $\sigma$. As $\sigma$ was an arbitrary
  ordering of~$V(C_{2k})$, this completes the proof.
\end{proof}

With Lemmas~\ref{lem:heavy} and~\ref{lem:lex copies} at hand, we
are ready to prove Lemma \ref{lem:multicolour-paths}.

\begin{proof}[Proof of Lemma~\ref{lem:multicolour-paths}] Fix $2\leq
  k\in \bbn$, $0<\rho\leq 1$ and
  $p=\omega\left(n^{-1+{1}/({2k-1})}\right)$ as in the statement of
  the lemma. Note that we can assume $\rho<1/4$ as the statement of
  the lemma is stronger for smaller $\rho$. We also define the
  following constants for $i=0,1,\ldots,k-1$:
  \begin{equation} \label{eq:constants}
    \gamma_i:=\frac{1}{4}\left(\frac{1}{20}\right)^i, \quad \theta:=
    \frac{\gamma_{k-1}}{2^{4k}}, \quad d:=\frac{1}{4}\cdot
    \left(\frac{\theta}{4}\right)^{2k}\cdot\prod_{i=1}^{k-1}\gamma^2_i,
    \quad \xi:=\rho^2d, \quad \mbox{ and } \quad \eps=\alpha:=\xi^2.
  \end{equation} 
  Recall from Definition~\ref{def:path count} that, given a graph $G$
  on vertex set $[n]$, a vertex $v$ and a pair of vertices $f$, we
  denote the number of paths with $\ell$ vertices in $G$ with
  endpoints $f$ by $X^\ell_G(f)$. Furthermore, the number of
  copies of $P_\ell$ in $G$ that have $v$ as an endpoint is denoted by
  $X^\ell_G(v)$.  Also, we will frequently use the concept of heavy
  vertices, sets, and colour allocations, as in
  Definition~\ref{def:heavy}.

  We let $\cG$ be the collection of $n$-vertex graphs $G$ satisfying
  the conclusions of Lemma \ref{lem:RG - properties}, which contains
  some simple properties of $G(n,p)$, of Lemma
  \ref{lem:well-distributed} with $\ell=2k$, concerning the number of
  paths between some pairs of vertices in $G(n,p)$, and also
  satisfying the two preparatory results, Lemma~\ref{lem:heavy} and
  Lemma~\ref{lem:lex copies} with $\nu:=\gamma_{k-1}$. Thus, a.a.s.\
  $\bfg\sim\bfg(n,p)$ is a subgraph of $\cG$ and so it suffices to show the desired conclusion for
  any $G\in \cG$.
		
  Fix some $G\in \cG$ and a colouring $\chi:E(G)\rightarrow \bbn$. Now
  assume that neither of the conclusions \ref{item:mono cycle} nor
  \ref{item:lexico} hold in our colouring and so we aim to prove that
  condition \ref{item:rainbow} holds, which means to prove that the
  rainbow focused graph $\Gamma = \Gamma(2k, G, \chi)$ is
  $(\rho,d)$-dense. In particular, we can assume that $\chi$ is
  monochromatic $C_{2k}$-free and so $G$ and $\chi$ satisfy~\ref{item:heavy1}, \ref{item:heavy2} and~\ref{item:heavy3} of
  Lemma~\ref{lem:heavy} for all $U$ and $\lambda$ as in that lemma. We
  claim that it is enough to prove the following claim.

  \begin{claim} \label{clm:many paths} For every vertex set $S\subset
    V(G)$ with $|S|=\rho n$, there are at least
    $2d\rho^2n^{2k}p^{2k-1}$ rainbow copies of~$P_{2k}$ in~$G$
    whose endpoints lie in~$S$.
  \end{claim}

  In order to see that Claim \ref{clm:many paths} suffices, note that
  by Remark \ref{rem:simple local dense}, in order to prove that
  $\Gamma=\Gamma(2k,G,\chi)$ is $(\rho,d)$-dense, we just need to show
  that every $S$ as in the claim contains at least $d(\rho n)^2/2$
  edges of $\Gamma$. Now, fixing such an $S$, for $f\in \binom{S}{2}$,
  let $Y(f)$ be the number of rainbow copies of $P_{2k}$ in $G$
  with endpoints~$f$. Let $F^+\subseteq
  \binom{S}{2}$ be the set of pairs of vertices~$f$ such that
  $Y(f)>0$ and $F^{++}\subseteq F^+$ be the set of pairs $f\in
  F^+$ such that $Y(f)> 2n^{2k-2}p^{2k-1}$.

  Let $F(G)$ be the collection of pairs $f\in \binom{[n]}{2}$ such
  that $X_G^{2k}(f)>2n^{2k-2}p^{2k-1}$. Since $Y(f)\leq X_G^{2k}(f)$
  for all pairs $f\in \binom{S}{2}$, we have that
  \begin{equation} \label{eq:sum of Ys} \sum_{f\in F^+\setminus
      F^{++}}Y(f)\geq \sum_{f\in F^+}Y(f)-\sum_{f\in
      F(G)}X_G^{2k}(f)\geq d\rho^2n^{2k}p^{2k-1},\end{equation} using
  the assumed conclusion of Lemma \ref{lem:well-distributed} in $G$,
  the conclusion of Claim \ref{clm:many paths} and our definition of
  $\xi$ \eqref{eq:constants} in the final inequality. By the
  definition of $F^{++}$, we also have that
  \[\sum_{f\in F^+\setminus F^{++}}Y(f)\leq
    2n^{2k-2}p^{2k-1}|F^+\setminus F^{++}|,\]
  and so combining with \eqref{eq:sum of Ys}, we get that $|F^+|\geq
  |F^+\setminus F^{++}|\geq d\rho^2n^2/2$. As $F^+$ is precisely the
  set of edges of $\Gamma$ on $S$, this concludes the proof of
  Lemma~\ref{lem:multicolour-paths} assuming Claim~\ref{clm:many
    paths}. 
\end{proof}
		
\begin{proof}[Proof of Claim~\ref{clm:many paths}]
  Fix some
  $S\subseteq V(G)$ with $|S|=\rho n$ and, for convenience,
  let~$U_0=S$.  Claim~\ref{clm:many paths} will be derived from the
  following claim that identifies the vertices and edges which will be
  used to form our desired rainbow paths. We will find disjoint vertex
  subsets $U_1,\ldots,U_{k-1}\subseteq V(G)\setminus U_0$ together
  with colour allocations satisfying some useful properties. Given
  $i=1,\dots,k-1$ and $u\in U_i$, an edge $e=uv$ is called
  \emph{good} if $v\in U_{i-1}$ and $\chi(e)\neq \lambda_{i-1}(v)$.

  \begin{claim} \label{clm:Ui} There exist disjoint vertex
    subsets $U_1,\ldots,U_{k-1}\subseteq V(G)\setminus U_0$ and
    colour allocations $\lambda_i:U_i\rightarrow \mathbb{N}\cup
    \{\star\}$ for $i=0,\dots,k-1$ such that the following
    conditions hold for $i=1,\dots,k-1$:
    \begin{enumerate}[{\rm(i)}]
    \item \label{item:Uisize} $|U_i|=\gamma_i n$.
    \item \label{item:lambdai} Either $\lambda_i(u)=\star$ for all
      $u\in U_i$ or $\lambda_i:U_i\rightarrow \bbn$ is an
      $\alpha$-heavy colour allocation.
    \item \label{item:good edges} For every $u\in U_{i}$, either
      \begin{enumerate}[{\rm(a)}]
      \item \label{type c} $\lambda_i(u)=c\in \mathbb{N}$ and $u$ is
        incident to at least $\theta p|U_{i-1}|$ good edges coloured
        $c$ by $\chi$; or
      \item \label{type star} $\lambda_i(u)=\star$ and for any
        collection $C\subseteq \bbn$ of at most $2k-1$ colours, we
        have that $u$ is incident to at least $\theta p|U_{i-1}|$ good
        edges $e$ such that $\chi(e)\notin C$.
      \end{enumerate}
    \item \label{item:bad edges} For any $u\in U_i$, we have that for
      any set $C\subseteq \mathbb{N}$ of at most $2k-1$ colours, $u$
      is incident to at most $\tfrac{\theta}{2}p|U_{i-1}|$ edges $uv$
      such that $v\in U_{i-1}$ and $\lambda_{i-1}(v)\in C$.
    \end{enumerate}
  \end{claim}
  Let us first see how Claim \ref{clm:Ui} implies Claim \ref{clm:many
    paths}. We will prove by induction that for $\ell=1,2,\ldots,k$
  there are at least
  \[\beta_\ell:=\left(\frac{\theta}{4}\right)^{2\ell}p^{2\ell-1}\prod_{j=1}^{\ell}|U_{k-j}|^2\]
  rainbow copies of $P_{2\ell}$ with endpoints in $U_{k-\ell}$. Our
  definition of constants from~\eqref{eq:constants} and the fact that
  $|U_i|=\gamma_i n$ (part \ref{item:Uisize} of Claim \ref{clm:Ui})
  then imply Claim \ref{clm:many paths} by applying the inductive
  statement with $\ell=k$. Now in order to prove the inductive
  statement, we will in fact only count rainbow copies $P$ of
  $P_{2\ell}$ with endpoints $u$, $v\in U_{k-\ell}$ that satisfy further
  the following conditions:
  \begin{itemize}
  \item $\lambda_{k-\ell}(u)=\lambda_{k-\ell}(v)$ only if
    $\lambda_{k-\ell}(u)=\lambda_{k-\ell}(v)=\star$, and
  \item $\lambda_{k-\ell}(u)$, $\lambda_{k-\ell}(v)\notin
    \chi(P):=\{\chi(e):e\in E(P)\}$.
\end{itemize}

  Let us first treat the base case $\ell=1$. By Claim \ref{clm:Ui}
  \ref{item:lambdai}, we have that either $\lambda_{k-1}(u)=\star$ for
  all $u\in U_{k-1}$ or $\lambda_{k-1}:U_{k-1}\rightarrow \bbn$ is an
  $\alpha$-heavy colour allocation. Suppose first that
  $\lambda_{k-1}(u)=\star$ for all $u\in U_{k-1}$. In that case, any
  edge in $G[U_{k-1}]$ gives a valid path $P_2$ and we have
  $(1-\eps)p|U_{k-1}|^2/2>\beta_1$ edges in $U_{k-1}$ by the assumed
  conclusion of Lemma~\ref{lem:RG - properties}~\ref{RG: edge count}
  in $G$.  On the other hand, if $\lambda_{k-1}:U_{k-1}\rightarrow
  \bbn$ is an $\alpha$-heavy colour allocation, then, using that
  $|U_{k-1}|\geq \nu n$ and that there are \emph{no} lexicographic
  copies of $C_{2k}$ for all orderings of its vertices (by assumption,
  otherwise we satisfy outcome \ref{item:lexico}), Lemma \ref{lem:lex
    copies} implies that there are $p|U_{k-1}|^2/2^{8k+1}>\beta_1$
  edges $e=uv$ in $G[U_{k-1}]$ such that $\chi(e)\notin
  \{\lambda(u),\lambda(v)\}$ and $\lambda(u)\neq \lambda(v)$. Each
  such edge gives a valid copy of $P_2$, as required.

  Now for the inductive step, suppose that $2\leq \ell \leq k$ and we
  have $\beta_{\ell-1}$ rainbow copies~$P$ of~$P_{2\ell-2}$ with
  endpoints $u,v\in U_{k-\ell+1}$ such that $\chi(P)\cap
  \{\lambda_{k-\ell+1}(u),\lambda_{k-\ell+1}(v)\}=\emptyset$ and
  $\lambda_{k-\ell+1}(u)=\lambda_{k-\ell+1}(v)$ only if
  $\lambda_{k-\ell+1}(u)=\lambda_{k-\ell+1}(v)=\star$. Fix any such
  copy $P$ with endpoints $u,v\in U_{k-\ell+1}$ and let $C=\chi(P)$.
  Furthermore, let $C'=C$ if $\lambda_{k-\ell+1}(v)=\star$ and
  $C'=C\cup \{\lambda_{k-\ell+1}(v)\}$ if $\lambda_{k-\ell+1}(v)\in
  \bbn$. Now by Claim \ref{clm:Ui} \ref{item:good edges} there are at
  least $\theta p|U_{k-\ell}|$ good edges $e=uw$ with $w\in
  U_{k-\ell}$ such that $\chi(e)\notin C'$. Indeed, if
  $\lambda_{k-\ell+1}(u)=c\in \bbn$, then we apply part \ref{type c}
  of the claim, using that $c\notin C'=\chi(P)\cup
  \{\lambda_{k-\ell+1}(v)\}$ by induction and if
  $\lambda_{k-\ell+1}(u)=\star$, we apply \ref{type star} directly to
  avoid colours in $C'=C$.  As these edges are good, we also have that
  $\lambda_{k-\ell}(w)\neq \chi(e)$. Moreover, by part \ref{item:bad
    edges} of Claim \ref{clm:Ui}, at most $\tfrac{\theta}{2}
  p|U_{k-\ell}|$ of these edges $e=uw$ have $\lambda_{k-\ell}(w)\in
  C'$ and so we have at least $\tfrac{\theta}{2} p|U_{k-\ell}|$
  choices of edge $e=uw$ such that $P\cup \{e\}$ is rainbow,
  $\lambda_{k-\ell}(w)\notin \chi(P\cup \{e\})=C\cup \{\chi(e)\}$ and
  $\lambda_{k-\ell+1}(v)\notin C'' $, where $C''$ is defined to be
  $C''=\chi(P\cup \{e\})$ if $\lambda_{k-\ell}(w)=\star$ and
  $C''=\chi(P\cup \{e\})\cup \{\lambda_{k-\ell}(w)\}$ if
  $\lambda_{k-\ell}(w)\in \bbn$.  Fixing such an edge $e=uw$, we now
  count the number of ways to extend the path $P\cup \{e\}$ to a copy
  of~$P_{2\ell}$ with the desired properties. By Claim \ref{clm:Ui}
  \ref{item:good edges}, we have that there are at least
  $\tfrac{\theta}{2} p|U_{k-\ell}|-1$ (the~$-1$ accounts for the fact
  that we have to avoid $w$) good edges $f=vz$ with $z\in
  U_{k-\ell}\setminus \{w\}$ such that $\chi(f)\notin C''$. Here we
  use that $\lambda_{k-\ell+1}(v)\notin C''$ if we are in case
  \ref{type c} and we apply the conclusion of \ref{type star} directly
  to avoid colour $C''$ edges if $\lambda_{k-\ell+1}(v)=\star$.  As
  these edges $f=vz$ are good, we also have that
  $\lambda_{k-\ell}(z)\neq \chi(f)$.  Moreover, by part \ref{item:bad
    edges} of Claim \ref{clm:Ui}, at most $\tfrac{\theta}{2}
  p|U_{k-\ell}|$ of these edges $f=vz$ have $\lambda_{k-\ell}(z)\in
  C''$ and so we have at least $\tfrac{\theta}{2}
  p|U_{k-\ell}|-1\geq\tfrac{\theta}{4} p|U_{k-\ell}| $ choices of edge
  $f=vz$ such that $P\cup \{e,f\}$ is a rainbow copy of $P_{2\ell}$
  with endpoints $w,z$ such that
  $\lambda_{k-\ell}(w),\lambda_{k-\ell}(z)\notin \chi (P\cup \{e,f\})$
  and $\{\lambda_{k-\ell}(w)\}\cap\{\lambda_{k-\ell}(z)\}\subseteq
  \{\star\}$. Summing over all choices of rainbow copies $P$ of
  $P_{2\ell-2}$ (appealing to induction) and choices of extending
  edges $e$ and $f$, we get at least
  $\beta_{\ell-1}\cdot\tfrac{\theta^2}{8}p^2|U_{k-\ell}|^2\geq
  \beta_\ell$ rainbow copies of $P_{2\ell}$ with the desired
  properties, completing the induction, and the proof of
  Claim~\ref{clm:many paths}.  
  \end{proof}

  It remains to prove Claim~\ref{clm:Ui}.

  \begin{proof}[Proof of Claim~\ref{clm:Ui}]
    We proceed by induction on $i=0,\ldots, k-1$.  At the time of
    fixing $U_i$ and $\lambda_i$ for 
  $i=0,\ldots,k-1$, we will ensure that $U_i$ satisfies conditions
  \ref{item:Uisize}--\ref{item:bad edges} of the claim. For
  $i=0,1,\ldots,k-2$, we will also impose an extra condition that will
  allow the induction to continue. Namely, we will guarantee the
  following extra condition holds:
  \textit{%
    \begin{enumerate}[{\rm(i)}]
      \setcounter{enumi}{4}
    \item \label{item:continue} For $i=0,\ldots,k-2$, fixing
      $Z_i=V(G)\setminus \left(\bigcup_{j=0}^iU_j\right)$, there are at
      least $ \gamma_i pn|U_i|$ edges $e=uz$ such that $u\in U_i$,
      $z\in Z_i$ and $\chi(e)\neq \lambda_i(u)$.
    \end{enumerate}}
  For our base case $i=0$, we let $\lambda_0(u)=\star$ for all $u\in
  U_0$, and check that this choice works.  Note that
  conditions~\ref{item:Uisize}--\ref{item:bad edges} of
  the claim are not required for $i=0$, and hence 
  it is only~\ref{item:continue} that we need
  to check.
  By the assumed conclusion of Lemma~\ref{lem:RG -
    properties}~\ref{RG:edges between subsets}, we have that there are
  at least 
  $(1-\eps)p|U_0||Z_0|\geq \gamma_0 pn|U_0|$ edges $e=uz$ with $u\in
  U_0$ and $z\in Z_0$, using here that we have $(1-\eps)|Z_0|\geq
  \gamma_0 n$ as $|U_0|=\rho n\leq n/4$. Moreover, each such edge
  $e=uz$ certainly has $\chi(e)\neq \lambda_0(u)$ as
  $\lambda_0(u)=\star$. This concludes the base case.
		
  For the induction step, suppose that $1\leq i\leq k-1$ and we have
  $U_0,\ldots, U_{i-1},\gamma_0,\dots,\gamma_{i-1}$ already defined
  satisfying all the conditions 
  of Claim~\ref{clm:Ui} and $Z_{i-1}:=V(G)\setminus
  \left(\bigcup_{j=0}^{i-1}U_j\right)$ such that there is a collection
  $E'\subseteq E(G)$ of at least $\gamma_{i-1} pn|U_{i-1}|$ edges
  $e=uz$ with $u\in U_{i-1}$, $z\in Z_{i-1}$, and $\chi(e)\neq
  \lambda_{i-1}(u)$ (condition \ref{item:continue}).  Now let
  $G'\subseteq G $ be the bipartite subgraph (with parts $U_{i-1}$ and
  $Z_{i-1}$) of $G$ defined by the set of edges $E'$. Let $Z'\subseteq
  Z_{i-1}$ be the set of vertices $z\in Z_{i-1}$ such that
  $d_{G'}(z,U_{i-1})\geq \tfrac{\gamma_{i-1}}{4}p|U_{i-1}|$. We claim
  that $|Z'|\geq \tfrac{5\gamma_{i-1}}{8}n$. Indeed, if this was not
  the case, then we would
  have 
  \begin{multline*}
    \qquad
    |E'|\leq e_{G'}(U_{i-1},Z_{i-1}\setminus Z')+e_{G}(U_{i-1},
    Z')\\
    \leq n\cdot
    \frac{\gamma_{i-1}}{4}p|U_{i-1}|+(1+\eps)p|U_{i-1}|\cdot\frac{5\gamma_{i-1}}{8}n<\gamma_{i-1}
    pn|U_{i-1}|,\qquad
  \end{multline*}
  contradicting our lower bound on the number of edges in $G'$. Here,
  we used the assumed conclusion of Lemma~\ref{lem:RG -
    properties}~\ref{RG:edges between subsets} in~$G$ to upper bound 
  $e_G(U_{i-1},Z')$: if $|U_{i-1}||Z'|\geq6n/\epsilon p$, then we have
  $e_G(U_{i-1},Z')\leq(1+\epsilon)p|U_{i-1}||Z'|<(1+\epsilon)p|U_{i-1}|(5\gamma_{i-1}/8)n$, 
  but if not, we consider a set~$\widetilde Z$ disjoint from~$U_{i-1}$
  with $Z'\subset\widetilde Z$ and $|\widetilde Z|=5\gamma_{i-1}n/8$
  and use that $e_G(U_{i-1},Z')\leq e_G(U_{i-1},\widetilde Z)$. 
  So we indeed have that $|Z'|\geq
  \tfrac{5\gamma_{i-1}}{8}n$ and by the assumed conclusion of Lemma~\ref{lem:heavy}~\ref{item:heavy1}, we have that there is a subset
  $Z''\subset Z'$ such that $|Z''|\geq
  \left(\tfrac{5\gamma_{i-1}}{8}-\alpha\right)n\geq
  \tfrac{\gamma_{i-1}}{2}n$ and for all $c\in \bbn$ and every $z\in
  Z''$, we have $|N_G(z)\cap \lambda_{i-1}^{-1}(c)|\leq \alpha pn$.  Indeed
  this is immediate if $\lambda_{i-1}^{-1}(\star)=U_{i-1}$ (we can
  take $Z''=Z'$) and if this is not the case, then condition
  \ref{item:lambdai} for $i-1$ (which is assumed by induction) implies
  that $\lambda_{i-1}:U_{i-1}\rightarrow \bbn$ is an $\alpha$-heavy
  colour allocation and we can apply Lemma~\ref{lem:heavy}~\ref{item:heavy1}.  Now we aim to find $U_{i}\subseteq Z''$ and we
  remark that this will imply that condition~\ref{item:bad edges}
  is satisfied. Indeed, this follows from the fact that
  $\alpha<\tfrac{\theta}{4k}$ and our condition on vertices in $Z''$.
		
  Now we define $\lambda_i:Z''\rightarrow \bbn\cup \{\star\}$ as
  follows. For $z\in Z''$, if there is some colour $c \in \bbn$ such
  that there are $\theta p |U_{i-1}|$ edges in $G'$ incident to $z$
  that are coloured $c$ by $\chi$, we define $\lambda_{i}(z)=c$. If no
  such colour exists, we set $\lambda_{i}(z)=\star$.  Now if
  $|\lambda_i^{-1}(\star)|\geq \gamma_i n$, we fix $U_i\subseteq
  \lambda_i^{-1}(\star)\subseteq Z''$ with $|U_i|=\gamma_i n$ and
  claim that all the conditions on $U_i$ and $\lambda_i$ (restricted
  to $U_i$) in Claim \ref{clm:Ui} as well as condition
  \ref{item:continue} (if $i\leq k-2$) are satisfied. Indeed,
  conditions \ref{item:Uisize} and \ref{item:lambdai} are immediate
  and we have already established \ref{item:bad edges} as
  $U_i\subseteq Z''$. For \ref{item:good edges}, we use that
  $\theta<\tfrac{\gamma_{i-1}}{8k}$. Indeed, for $u\in U_i$ and any
  set $C\subseteq \bbn$ of size at most $2k-1$, we have that $u$ is
  incident to at most $(2k-1)\theta p |U_{i-1}|$ edges of $G'$ that
  are coloured with colours in~$C$. Hence, as $d_{G'}(u,U_{i-1})\geq
  \tfrac{\gamma_{i-1}}{4} p|U_{i-1}|$ (as $U_i\subseteq Z'$), we have
  that there are at least $\theta p|U_{i-1}|$ edges of $G'$ incident
  to $u$ which avoid the colours in~$C$. As each of these edges $uv\in
  E(G')$ with $v\in U_{i-1}$ has $\chi(e)\neq \lambda_{i-1}(v)$,
  condition \ref{item:good edges} of Claim \ref{clm:Ui} follows. If
  $i\leq k-2$, in order to establish \ref{item:continue}, we appeal to
  the assumed conclusion of Lemma~\ref{lem:RG - properties}~\ref{RG:edges between subsets} in $G$ which gives that
  $e_G(U_i,Z_i)\geq (1-\eps)p|U_i||Z_i|\geq \gamma_i pn|U_i|$ as
  $(1-\eps)|Z_i|\geq \gamma_i n$ because of
  condition~\ref{item:Uisize} on the sizes of $U_0,U_1,\ldots,U_{i}$
  which have  
  been established so far. As any edge $e=uz\in E(G)$ with $u\in U_i$ and
  $z\in Z_i$ has $\chi(e)\neq\lambda_i(u)=\star$, all the edges in~$G$
  between~$U_i$ and~$Z_i$ are valid for condition~\ref{item:continue}
  and we are
  done. 

  Thus it remains to treat the case when $|\lambda_i^{-1}(\star)|\leq
  \gamma_i n$. In this case, we let $U\subseteq Z''\setminus
  \lambda_i^{-1}(\star)$ such that $|U|=\gamma_{k-1} n$ if $i=k-1$
  and $|U|=8\gamma_{i} n$ if $i\leq k-2$, which is possible as
  $\gamma_i<\tfrac{\gamma_{i-1}}{18}$. 
  Now note that every vertex $u\in U$ has at least $\theta p
  |U_{i-1}|\geq2\alpha p n\geq\alpha d_G(u)$ incident edges in colour
  $\lambda_i(u)$ 
  and so~$\lambda_i$ (restricted to~$U$) is an $\alpha$-heavy colour
  allocation. Moreover, we certainly have that $|U|\geq \alpha n$ and
  so the assumed conclusion of Lemma~\ref{lem:heavy}~\ref{item:heavy3} gives some subset $W\subseteq U$ as in that
  lemma. Now if $i=k-1$, we fix $U_i=U$ and if $i\leq k-2$ we fix
  $U_i=W$. Let us check the conditions of Claim \ref{clm:Ui} (and
  condition \ref{item:continue}). By our choice of $U_i$, we certainly
  have that conditions \ref{item:Uisize} and \ref{item:lambdai} (with
  $\lambda_i$ restricted to $U_i$) both hold and we have already
  established condition \ref{item:bad edges} as $U_i\subseteq Z''$. To
  see condition \ref{item:good edges}, note that each $u\in U_i$ is
  incident to at least $\theta p|U_{i-1}|$ edges in $G'$ which are
  coloured $\lambda_i(u)$, by our definition of $\lambda_i$ and the
  fact that $\lambda^{-1}_i(\star)\cap U_i=\emptyset$. Moreover
  all the edges $e=uv\in E(G')$ with $u\in U_i$ are \emph{good} as
  defined just before the statement of Claim~\ref{clm:Ui}, 
  as by the definition of
  $G'$ they have $v\in U_{i-1}$ and $\chi(e)\neq \lambda_{i-1}(v)$.
		
  It remains to establish \ref{item:continue} in the case that $i\leq
  k-2$.
  We defined $U_i=W\subseteq U$ from our application of Lemma~\ref{lem:heavy}~\ref{item:heavy3} and so we have that
  $|U_i|=\tfrac{1}{8}|U|$ and each vertex $u\in U_i$ is incident to at
  least $p|U|/5\geq p|U_{i}|$ edges $e=uv$ with $v\in U\setminus U_i$
  and $\chi(e)\neq \lambda_i(u)$. As $U\setminus U_i\subseteq
  Z_i:=V(G)\setminus \left(\bigcup_{j=0}^iU_j\right)$, summing over all
  $u\in U_i$, we get at least $p|U_i|^2\geq \gamma_i pn|U_i|$ edges
  $e=uz$ with $u\in U_i$, $z\in Z_i$ and $\chi(e)\neq \lambda_i(u)$,
  establishing \ref{item:continue}, completing the induction step and
  concluding the proof of Claim \ref{clm:Ui}.
\end{proof}
	
	
	
	
	

   

	
	
\bibliographystyle{amsplain}
\bibliography{extracted-mrefed}
	

	%
	
	
	
	
\appendix
	
\section{Proof of Lemma~\ref{lem:orient}}
\label{sec:pfoforient}
	
	\def\vG{\vec G}
	\def\vR{\vec R}
	\def\vC{\vec C}
	\def\vK{\vec K}
	\def\vJ{\vec J}
	\def\ve{\vec e}
	
	\def\ora#1{\overrightarrow{#1}}
	\def\cP{\mathcal{P}}  
	\def\cW{\mathcal{W}}

	Let~$\vG$ be an oriented graph, and let~$A$ and~$B\subset
        V(\vG)$ be disjoint. We write~$\vG[A,B]$ for the oriented
        graph with vertex set~$A\cup B$ and arc set $\{\ora{ab}\in
        E(\vG) : a\in A,\,b\in B \}$. Set $e_{\vG}(A,B) :=
        e(\vG[A,B]))$.
        Define the \emph{$p$-density} of $(A,B)$ in $\vG$ as 
	\begin{linenomath}
          \begin{align*}
			d_{\vG,p}(A,B) := \frac{e_{\vG}(A,B)}{p|A||B|}.
          \end{align*}
        \end{linenomath}
	Let us say that
	$\vG[A,B]$ is \emph{$(\delta,p)$}-regular if
	\begin{linenomath} \begin{align*}
			|d_{\vG,p}(X,Y)-d_{\vG,p}(A,B)|\leq\delta
	\end{align*} \end{linenomath}
	for every $X\subset A$ and~$Y\subset B$ with $|X|\geq\delta|A|$
	and $|Y|\geq\delta|B|$.
	
	Given an $n$-vertex oriented graph $\vG$ and $\delta>0$, a
        partition $\{V_0,V_1,\ldots,V_k\}$ of the vertex set of $\vG$
        is called \emph{$(\delta,k)$-equitable} if $|V_0|\leq \delta
        n$ and $|V_1|=\ldots = |V_k|$. Moreover, we say that a
        $(\delta,k)$-equitable partition is
        \emph{$(\delta,\vG,p)$}-regular if for all, but at most
        $\delta k^2$, of the distinct pairs $(i,j)\in [k]^2$ the
        oriented graphs $\vG[V_i,V_j]$ are
        $(\delta,p)$-regular. Finally, say that an $n$-vertex graph
        $G$ with density $p$ is \emph{$(\xi,D,p)$-upper uniform} if we
        have
	\begin{linenomath}\begin{align*}
			e_{G}(X,Y) \leq D p |X| |Y|
	\end{align*} \end{linenomath}
      for all disjoint vertex subsets $X$ and $Y$ with $|X| \geq \xi n$ and $|Y| \geq \xi n$.
	
      Now, we state a well-known version of the Sparse Szemer\'edi
      Regularity Lemma (abbr. SRL) for oriented graphs. This version
      can be easily deduced from its non-oriented
      counterparts~\cite{K97,KR03}.
	
      \begin{theorem}[SRL version for oriented
        graphs]\label{thm:SRL-or}
        For every $\delta,D>0$ and every integer $m\geq 1$ there are
        $\xi>0$ and an integer $M\geq m$ such that every (sufficiently
        large) $(\xi,D,p)$-upper uniform graph $G$ has the following
        property. For any orientation $\vG$ of $G$, there exists a
        $(\delta,\vG,p)$-regular partition into $k+1$ classes with
        $m\leq k \leq M$.
	\end{theorem} 
	
	Analogous to the non-oriented version of the (sparse)
        regularity method, we define the `reduced' (dense) structure
        associated with a regular partition. Let $\cP =
        \{V_0,\ldots,V_k\}$ be a $(\delta,\vG ,p)$-regular partition
        and let $\theta>0$ be a parameter. The reduced digraph $\vR =
        \vR(\cP,\theta;\delta,\vG,p)$ associated with $\cP$ has vertex
        set $[k]$ and for every distinct pair $(i,j)\in [k]^2$ we draw
        the arc $\ora{ij}$ in $\vR$ if and only if $\vG[V_i,V_j]$ is
        $(\delta,p)$-regular and $d_{\vG,p}(V_i,V_j) \geq \theta$.
	
	Let us now define graphs of a certain type whose appearance
        in~$G$ implies the existence of many cycles in~$G$, which, in
        turn, can be made to imply the existence of many copies of the
        orientations~$\vec C_\ell$ of~$C_\ell$ that we are after in
        Lemma~\ref{lem:orient}.  Let~$J$ be a graph with a given
        vertex partition $V(J)=\bigcup_{0\leq i<\ell}U_i$, with
        $|U_i|=u$ for every $i$.
	In what follows, the indices of the sets~$U_i$ will be
        considered modulo~$\ell$.  We further assume that
        $E(J)=\bigcup_{0\leq i<\ell}E(U_{i-1},U_i)$, that is, the
        edges of~$J$ connect vertices in consecutive sets~$U_i$.
        Thus~$J$ has a `cyclic' structure.  Finally, assume that
        $J[U_{i-1},U_i]$ is $(\delta,p)$-regular\footnote{Here, we
          refer to the undirected sparse-regular notion.} for every
        $0\leq i<\ell$ and that $d_p(U_{i-1},U_i)\geq\theta$. We then
        say that~$J$ is a \emph{$(C_{\ell},\theta,\delta,u,p)$-regular
          system}.  A crucial fact is that `large'
        $(C_{\ell},\theta,\delta,u,p)$-regular systems that occur as
        subgraphs of random graphs contain many copies of~$C_\ell$, as
        long as the involved parameters are appropriate.  In order to
        state this more precisely, we denote by $\iota_{\ell}(J)$ the
        number of transversal\footnote{That is, contains one vertex
          from each vertex class of~$J$.} copies of $C_{\ell}$ in
        $J$. The following lemma is a straightforward consequence of
        Theorem 1.6
        in~\cite{CGSS14}.
	\begin{lemma}[K{\L}R conjecture for cycles in random graphs:
          counting version]
          \label{lem:cycles_in_J}
          Let $p=p(n)=\omega(n^{-1+1/(\ell-1)})$ where
          $3\leq\ell\in\bbn$.  For any~$\theta>0$ there are~$\zeta>0$
          and~$\delta>0$ such that for any $\mu >0$ a.a.s.\
          $\bfg\sim\bfg(n,p)$ has the following property. For every
          $u\geq \mu n$ and every
          $(C_{\ell},\theta,\delta,u,p)$-regular system $J \, \subset
          \, \bfg$ we have that $\iota_{\ell}(J) \geq \zeta \, (\mu p
          n)^{\ell} $.
	\end{lemma}
	
	Now, we are able to prove Lemma~\ref{lem:orient}, which we
        remind here for convenience. Given $3\leq\ell\in \bbn$ and
        $\epsilon>0$, we want to show that the following holds for
        $p=\omega\big(n^{-1+1/(\ell-1)}\big)$: there is~$\eta>0$ such
        that for every~$\nu>0$ a.a.s.\ $\bfg\sim \bfg(n,p)$ has the
        property that for every $W\subseteq V(\bfg)$ with $|W|\geq \nu
        n$, every subgraph~$G\subseteq \bfg[W]$ with
        \begin{equation}
          \label{eq:1b}
          e(G)\geq\left(\pi(K_{2^{\ell-1}})+\eps\right)
          p\binom{|W|}2      
        \end{equation}
        has the $\big(C_\ell,\eta(\nu pn\big)^\ell)$-orientation property.
        
	\begin{proof}[Proof of Lemma~\ref{lem:orient}]
          Of course, we can assume that $0<\eps < 1$.
          Given the parameters $\ell \geq 3$, $\eps>0$ and $\nu>0$,
          and setting $\eps' := \frac{\eps}{72}$, we choose the
          parameters $\theta = \theta(\eps) := \eps'$ and $D = D(\eps)
          := 1 + \eps'$. Applying Lemma~\ref{lem:cycles_in_J} we set
          $\delta = \delta(\ell,\eps) :=
          \min\{\delta_{\ref{lem:cycles_in_J}}(\ell,\theta) , \eps'
          \}$ and $\zeta:=\zeta(\ell,\eps) :=
          \zeta_{\ref{lem:cycles_in_J}}(\ell,\theta)$. Applying
          Theorem~\ref{thm:SRL-or} for $m = m(\ell,\eps) := \max{\{
            N(\ell,\eps),1/{\eps}' \} }$\footnote{ Recalling the
            definition of $\pi(H)$ from~\eqref{eq:turan_density_def},
            we define $N(\ell,\eps)$ large enough such that every
            graph $H$ with $v(H) \geq N(\ell, \eps)$ vertices and
            $e(H) \geq ( \pi(K_{2^{\ell-1}}) + \eps/2 )
            \binom{v(H)}{2}$ contains a copy of $K_{2^{\ell-1}}$. },
          we set $\xi = \xi(\ell,\eps) :=
          \xi_{\ref{thm:SRL-or}}(\delta,D,m)$ and $M = M(\ell,\eps) :=
          M_{\ref{thm:SRL-or}}(\delta,D,m)$. Now, we define the
          parameters $\mu := \frac{1-\delta}{M} \, \nu$ and
          \begin{linenomath}
            \begin{align*} 
              \eta = \eta(\ell,\eps) \:= \zeta \, \left(\frac{1-\delta}{M} \right)^{\ell}.
            \end{align*}
          \end{linenomath}
          Of course, this $\eta$ was defined in order to be our
          desired parameter in Lemma~\ref{lem:orient}.
		
          Let $p = \omega(n^{-1+1/(\ell-1)})$. We fix any outcome of
          $\bfg$ which satisfies the following properties. Set $\mu'
          := \min\{ \xi, \mu
          \}$. 
          \begin{itemize}
          \item[(i)] For each $U \in \binom{n}{ \geq \mu' n }$ we have
            that
            \begin{linenomath} \begin{align*} e_{\bfg}(U) = (1 \pm
                \eps') p |U|^2/2.
            \end{align*} \end{linenomath}
        \item[(ii)] For every disjoint sets $U,W \in \binom{n}{ \geq
            \mu' n }$ we have that
          \begin{linenomath} \begin{align*} e_{\bfg}(U,W) = (1 \pm
              \eps') p |U||W|.
          \end{align*}\end{linenomath}
      \item[(iii)] For every $u\geq \mu n$ and every
        $(C_{\ell},\theta,\delta,u,p)$-regular system $J\subseteq
        \bfg$ we have that
        \begin{linenomath} \begin{align*} \iota_{C_{\ell}}(J) \geq
            \zeta \, (\mu p n)^{\ell}.
        \end{align*} \end{linenomath}
    \end{itemize}
    Note that by Lemma~\ref{lem:RG - properties} ($1$) and ($3$), and
    Lemma~\ref{lem:cycles_in_J} applied to these parameters (together
    with a simple union bound), we conclude that these properties
    ($i$), ($ii$) and $(iii)$ a.a.s.\ occur. In view of this, the rest
    of the proof is deterministic (and we take $n$ sufficiently large
    as necessary).
		
    Thus, fix $W \in \binom{V(\bfg)}{\geq \nu n}$ and $G\subseteq
    \bfg[W]$ with
    \begin{linenomath}
      \begin{align*}
        e(G) \geq (\pi(\ell) + \eps) p \binom{|W|}{2}.
      \end{align*}
    \end{linenomath}
    where $\pi(\ell) := \pi\left(K_{2^{\ell-1}}\right)$.
    Our aim is to show that $G$ has the $(C_{\ell}, \eta(\nu
    pn)^{\ell})$-property. Next, fix an arbitrary orientation of $\vG$
    and an acyclic orientation $\vC_{\ell}$ of $C_{\ell}$. We want to
    show that $\vG$ contains at least $\eta (\nu pn)^{\ell}$ copies of
    $\vC_{\ell}$. Note that $G$ is $(\xi,D,p)$-upper uniform, more
    precisely $G$ is $(\mu',D,p)$-upper uniform because $\bfg$
    satisfies such property by ($ii$) and $G\subseteq \bfg$. Then, we
    apply Theorem~\ref{thm:SRL-or} (with our choice of parameters) to
    conclude that there exists a $(\delta, \vG,p)$-regular partition
    $\cW = \{ W_0,W_1,\ldots,W_k \}$ of $W$ with $k\in [m,M]$. Now,
    let $\vR = \vR(\cW,\theta;\delta,\vG,p)$ be the reduced digraph
    associated with $\cW$. Note that $|W_i| \geq un$ for each $i\in
    [k]$, where $u := \frac{|W|-|W_0|}{k}$. Clearly, $u\geq
    \frac{(1-\delta)\nu n}{M} = \mu n$.
		
    The key observation is the following `transference' claim.
    \begin{claim}[Edge-density Inheritance]
      \label{claim:inh}
      Let $\vG$ and $\vR$ be as above. Then,
      \begin{linenomath} \begin{align}\label{eq:transfer-dense} e(R)
          \geq (\pi(\ell) + \eps/2) \binom{k}{2}.
      \end{align} \end{linenomath} 
    In particular, $R$ contains a copy of $K_{2^{\ell-1}}$. 	
  \end{claim}
  \begin{proof}
    First, note that the second part of the claim is a straightforward
    consequence of \eqref{eq:transfer-dense} together with the fact
    that $k$ is large enough, specifically $k \geq N(\ell,\eps)$.	
    Thus, in order to show \eqref{eq:transfer-dense}, we proceed by
    contradiction. Suppose that $e(R) < (\pi(\ell) + \eps/2)
    \binom{k}{2}$ and set $\tilde{n} := |W|$.
    Then, since $G$ is $(\mu,D,p)$-upper uniform and $\mu n \leq |W_i|
    \leq \tilde{n}/k $ for $i\in [k]$, we obtain that
    \begin{linenomath}
      \begin{align*}
        \sum_{ij\in R}{e_G(W_i,W_j)} & \leq  \sum_{ij\in R}{(1+\eps')p|W_i||W_j|} \\
                                     & \leq  (1+\eps')p \tilde{n}^2/k^2 \times e(R) \\
                                     & < (1+\eps')p \tilde{n}^2/k^2 \times (\pi(\ell) + \eps/2) k^2/2  \\
                                     & \leq (\pi(\ell) + \eps/2 + 2
                                       \eps') p\tilde{n}^2/2.
      \end{align*}
    \end{linenomath}
    On the other hand,
    \begin{linenomath}
      \begin{align}
         \sum_{ij\notin R}{e(G[W_i,W_j])} & = \sum_{ij\in \binom{k}{2}}{(e_{\vG}(W_i,W_j) + e_{\vG}(W_j,W_i)) \mathds{1}[{ij\notin R}}] \nonumber \\
        & \leq \sum_{ij\in \binom{k}{2}}{e_{\vG}(W_i,W_j) \mathds{1}[{\ora{ij} \notin \vR }] + e_{\vG}(W_j,W_i) \mathds{1}[{\ora{ji} \notin \vR } }] \nonumber \\
        & \leq  \sum_{ij\in \binom{k}{2}}{e_{\vG}(W_i,W_j) \mathds{1}[{d_{\vG,p}(W_i,W_j) < \theta }] + e_{\vG}(W_j,W_i) \mathds{1}[{d_{\vG,p}(W_j,W_i) < \theta}}]\nonumber \\
        & \quad + \sum_{ij\in \binom{k}{2}}{ e_{\vG}(W_i,W_j)
           \mathds{1}[{\vG[W_i,W_j] \mbox{ is not $(\delta,p)$-regular}
           }} ] \nonumber
        \\
        & \quad + \quad e_{\vG}(W_j,W_i) \mathds{1}[{ \vG[W_j,W_i] \mbox{ is not $(\delta,p)$-regular}  }] \nonumber \\
        & \leq k^2 \theta p \tilde{n}^2/k^2 \nonumber \\
        & \quad + 2 (1+\eps') p\tilde{n}^2/k^2 \sum_{ij\in
           \binom{k}{2}}{ \mathds{1}[{\vG[W_i,W_j] \mbox{ or }  \vG[W_j,W_i] \mbox{ is not $(\delta,p)$-regular} } }  ] \nonumber \\
        & \leq \theta p\tilde{n}^2 + 2 (1+\eps') \delta  p\tilde{n}^2 
           \leq (\theta + 4\delta) p\tilde{n}^2\nonumber \\
        & \leq 10 \eps' \frac{p\tilde{n}^2}{2}.
      \end{align} \end{linenomath} 
    So, summing both estimations, we are showed that $e(G[\cW \setminus \{W_0\}]) \leq (\pi(\ell) + \eps/2 + 12 \eps') p\tilde{n}^2/2$. 
    Finally, due to property ($i$) of $\bfg$, we conclude that
    \begin{linenomath}
      \begin{align*}
      e(G) & = e(G[\cW\setminus\{W_0\}]) + \sum_{i=1}^{k}{e(G[W_i])} +
             (  e(G[W_0]) + e(G[W_0,W\setminus W_0]) ) \\
           & \leq e(G[\cW\setminus\{W_0\}]) +
             \sum_{i=1}^{k}{e(\bfg[W_i])} + (  e(\bfg[W_0]) + e(\bfg[W_0,W\setminus W_0]) ) \\
           &  =  e(G[\cW\setminus\{W_0\}]) +
             \sum_{i=1}^{k}{e(\bfg[W_i])} + (e(\bfg[W]) - e(\bfg[W-W_0]) )\\
           & \leq (\pi(\ell) + \eps/2 + 12 \eps') \frac{p\tilde{n}^2}{2} + k (1+ \eps') \frac{ p \tilde{n}^2}{2 k^2} + \left((1+\eps') \frac{p \tilde{n}^2}{2} - (1-\eps') (1-\delta)^2 \frac{ p \tilde{n}^2}{2} \right) \\
           & \leq  (\pi(\ell) + \eps/2 + 2/m + 16 \eps') \frac{p\tilde{n}^2}{2} \\
           & \leq (\pi(\ell) + \eps/2 + 18 \eps' ) \frac{p\tilde{n}^2}{2} \\
           & \leq (\pi(\ell) + 3 \eps/4) \frac{p\tilde{n}^2}{2},		
      \end{align*}
    \end{linenomath}
    which is the desired contraction.
  \end{proof}
		
  By Claim~\ref{claim:inh}, we conclude that $\vR$ contains an
  orientation $\vK_{2^{\ell-1}}$ of $K_{2^{\ell-1}}$.
  Now we show that we can find a copy of any acyclic orientation of
  $C_{\ell}$ in $\vK_{2^{\ell-1}}$. To see this, it is enough to show
  that any orientation of $K_{2^{\ell-1}}$ contains a transitive
  tournament with $\ell$ vertices (denoted as $TK_{\ell}$), since all
  acyclic orientations of $C_{\ell}$ lie inside $TK_{\ell}$. Indeed,
  we proceed by induction over $\ell\geq 1$. The base case $\ell=1$ is
  trivial. Suppose that it holds for $\ell-1$, and pick a vertex in
  $V(\vK_{2^{\ell-1}})$, say $v$. By the pigeonhole principle we can
  assume (w.l.o.g.) that $|N^{\rm in}(v)| \geq 2^{\ell-2}$. Thus, by
  the induction hypothesis, one can find a copy of $TK_{\ell-1}$
  within $N^{\rm in}(v)$. Then, by connecting this copy with the
  vertex $v$, we build the desired copy of $TK_{\ell}$.
		       
  Note that such $\vC_{\ell}$ in $\vR$ define (by lifting) a
  $(C_{\ell}, \theta, \delta,u,p)$-regular system $J \subseteq \bfg$
  with a natural orientation $\vJ \subseteq \vG$ of $J$ which obeys
  the behaviour of $\vC_{\ell}$.
		
  Finally, by property ($iii$), we conclude that
  \begin{linenomath}
    \begin{align*}
      \iota_{C_{\ell}}(J) \geq \zeta (\mu p n)^{\ell} =
      \frac{\zeta (1-\delta)^\ell \nu^{\ell}}{M^\ell} (pn)^{\ell} =
      \eta (\nu pn)^{\ell}.
    \end{align*}
  \end{linenomath}     	  
  Therefore, since each transversal $\ell$-cycle in $J$ induces a copy
  of $\vC_{\ell}$ in $\vJ \subseteq \vG$, the proof is complete.
\end{proof}

\end{document}